\documentclass[dvipdfmx]{amsart}
\usepackage{amsmath,amsthm,amssymb}

\makeatletter
    
    \@addtoreset{equation}{section}
  \makeatother

\usepackage{amsfonts}
\usepackage{mathrsfs}
\usepackage{bbm}

\usepackage[setpagesize=false,colorlinks=true,linkcolor=blue,citecolor=blue]{hyperref}

\newtheorem{definition}{Definition}[section]
\newtheorem{proposition}[definition]{Proposition}
\newtheorem{theorem}[definition]{Theorem}
\newtheorem{lemma}[definition]{Lemma}
\newtheorem{remark}{Remark}[section]

\pagebreak

\title[Semilinear damped wave equation]{
Critical exponent for the semilinear wave equations with a
damping increasing in the far field
}
\author[K. Nishihara]{Kenji Nishihara}
\address[K. Nishihara]{
Professor Emeritus, Waseda University, Tokyo 169-8050, Japan}
\email{kenji@waseda.jp}

\author[M. Sobajima]{Motohiro Sobajima}
\address[M. Sobajima]{
Department of Mathematics, Faculty of Science and Technology,
Tokyo University of Science, 2641 Yamazaki, Noda-shi, Chiba, 278-8510, Japan}
\email{msobajima1984@gmail.com}

\author[Y. Wakasugi]{Yuta Wakasugi}
\address[Y. Wakasugi]{
Department of Engineering for Production and Environment,
Graduate School of Science and Engineering,
Ehime University,
3 Bunkyo-cho, Matsuyama, Ehime, 790-8577, Japan
}
\email{wakasugi.yuta.vi@ehime-u.ac.jp}
\begin{document}
\begin{abstract}
We consider the Cauchy problem of the semilinear
wave equation with a damping term
\begin{align*}
	\left\{ \begin{array}{ll}
	u_{tt} - \Delta u + c(t,x) u_t = |u|^p,&(t,x)\in (0,\infty)\times \mathbb{R}^N,\\
	u(0,x) = \varepsilon u_0(x), \quad u_t(0,x) = \varepsilon u_1(x),& x\in \mathbb{R}^N,
	\end{array}\right.
\end{align*}
where
$p>1$
and the coefficient of the damping term has the form
\begin{align*}
	c(t,x) = a_0 (1+|x|^2)^{-\alpha/2} (1+t)^{-\beta}
\end{align*}
with some
$a_0 > 0$, $\alpha < 0$, $\beta \in (-1, 1]$.
In particular, we mainly consider the cases
\begin{align*}
	\alpha < 0, \beta =0 \quad \mbox{or} \quad \alpha < 0, \beta = 1,
\end{align*}
which imply
$\alpha + \beta < 1$,
namely, the damping is spatially increasing and effective. 
Our aim is to prove that the critical exponent is given by
\begin{align*}
	p = 1+ \frac{2}{N-\alpha}.
\end{align*}
This shows that the critical exponent is the same as that of the
corresponding parabolic equation
\begin{align*}
	c(t,x) v_t - \Delta v = |v|^p.
\end{align*}
The global existence part is proved by
a weighted energy estimates with an exponential-type weight function
and a special case of the Caffarelli-Kohn-Nirenberg inequality.
The blow-up part is proved by
a test-function method
introduced by Ikeda and Sobajima \cite{IkeSopre1}.
We also give an upper estimate of the lifespan.
\end{abstract}

\maketitle
\section{Introduction}
\footnote[0]{2010 Mathematics Subject Classification. 35L15; 35A01, 35B44}

We consider the Cauchy problem of the semilinear
wave equation with a damping term
\begin{align}
\label{dw}
	\left\{ \begin{array}{ll}
	u_{tt} - \Delta u + c(t,x) u_t = |u|^p,&(t,x)\in (0,\infty)\times \mathbb{R}^N,\\
	u(0,x) = \varepsilon u_0(x), \quad u_t(0,x) = \varepsilon u_1(x),& x\in \mathbb{R}^N,
	\end{array}\right.
\end{align}
where
$N \ge 1$,
$u = u(t,x)$
is a real-valued unknown function,
$\varepsilon$
is a small positive parameter,
$u_0, u_1$
are given initial data,
$p>1$,
and the coefficient of the damping term has the form
\begin{align}
\label{damping}
	c(t,x) = a(x)b(t) = a_0 \langle x \rangle^{-\alpha} (1+t)^{-\beta},
\end{align}
where $\langle x \rangle := \sqrt{1+|x|^2}$, $a(x) = a_0 \langle x \rangle^{-\alpha}$, $b(t) = (1+t)^{-\beta}$,
with some
$a_0 > 0$, $\alpha < 2$ and $\beta > -1$.
In particular, in this paper we mainly consider the cases
\begin{align}
\label{beta}
	\alpha<0, \beta =0\quad \mbox{or} \quad \alpha < 0, \beta = 1.
\end{align}
We assume that the initial data satisfies
\begin{align}
\label{ini}
	u_0 \in H^1(\mathbb{R}^N), \quad
	u_1 \in L^2(\mathbb{R}^N),\quad
	{\rm supp\,}(u_0, u_1) \subset \{ x \in \mathbb{R}^N ; |x| \le R_0\}
\end{align}
with some
$R_0>0$.

Our aim is to determine the critical exponent
$p_c$.
Here, the meaning of the {\em critical exponent} is the following:
if
$p > p_c$, then for any $(u_0, u_1)$ satisfying \eqref{ini},
there exists a unique global solution for sufficiently small
$\varepsilon$;
if
$p \le p_c$, then there exists $(u_0, u_1)$ satisfying \eqref{ini},
the local solution blows up in finite time for any small $\varepsilon$.
In this paper, we will show that
under the conditions \eqref{damping} and \eqref{beta},
the critical exponent is determined by
\begin{align*}
	p_c = 1+ \frac{2}{N-\alpha}.
\end{align*}
Comparing with previous studies we will explain below,
our novelties are to determine the critical exponent for
spatially increasing damping, and
to give the blow-up of solutions for the damping depending on
the time and space variables.

The Cauchy problem of the linear damped wave equation
\begin{align}
\label{ldw}
	\left\{ \begin{array}{ll}
	u_{tt} - \Delta u + c(t,x) u_t = 0,&(t,x)\in (0,\infty)\times \mathbb{R}^N,\\
	u(0,x) = \varepsilon u_0(x), \quad u_t(0,x) = \varepsilon u_1(x),& x\in \mathbb{R}^N,
	\end{array}\right.
\end{align}
with a damping coefficient
$c(t,x) = a(x)b(t) = a_0\langle x \rangle^{-\alpha}(1+t)^{-\beta}$
has been studied for a long time.
Roughly speaking, it is known that
if the damping is sufficiently strong, in other words,
effective, then the solution behaves like that of the corresponding parabolic equation
$c(t,x)u_t -\Delta u = 0$ (diffusion phenomenon).
On the other hand, if the damping is sufficiently weak, in other words, non-effective,
then the solution behaves like that of the wave equation without damping (scattering).

It is known that the classical damping $\alpha=\beta=0$
is included in the effective case,
and the diffusion phenomenon was studied by
\cite{Ma76, HsLi92, Ni97, Kar00, YaMi00, Ni03MathZ, MaNi03, HoOg04, Na04, ChHa03, SaWa17, Mi}.

On the other hand, Yamazaki \cite{Ya06} and Wirth \cite{Wi04, Wi06, Wi07JDE, Wi07ADE}
considered time-dependent damping $\alpha =0, \beta \in \mathbb{R}$,
and classified the behavior of the solution in the following way:
(i) Scattering: if $\beta > 1$, then the solution behaves like that of the wave equation without damping;
(ii) Scale-invariant weak damping: if $\beta =1$, then the asymptotic behavior of the solution
depends on $a_0$;
(iii) Effective: if $-1\le \beta <1$, then the solution behaves like that of the corresponding
parabolic equation;
(iv) Overdamping: if $\beta < -1$, then the solution does not decay to zero in general.

The space-dependent damping $\alpha \in \mathbb{R}, \beta =0$
was also studied by
\cite{RauTa74, Mo76, Ik05IJPAM, ToYo09, IkToYo13, Wa14, SoWa16, SoWa17, SoWa18, SoWa},
and similarly to the above, the behavior of the solution was classified in the following way:
(i) Scattering: if $\alpha >1$, then the solution behaves like that of the wave equation without damping;
(ii) Scale-invariant weak damping: if $\alpha =1$, then the asymptotic behavior of the solution
depends on $a_0$;
(iii) Effective: if $\alpha <1$, then the solution behaves like that of the corresponding
parabolic equation. We note that in the space-dependent case, the overdamping phenomenon
does not occur.

In a similar approach to the space-dependent case,
these results are partially extended to the space-time dependent damping
$\alpha, \beta \in \mathbb{R}$.
Mochizuki and Nakazawa \cite{MoNa96} proved that
the case $\alpha, \beta \ge 0, \alpha+\beta >1$ belongs to the scattering case.
For $0\le \alpha <1, -1<\beta <1$, $0<\alpha+\beta<1$,
by \cite{KeKe11, Kh11}, energy estimates of solutions were obtained and
they indicate the solution has diffusion phenomenon.


Based on the studies on the linear problem,
recently, the Cauchy problem of the semilinear damped wave equation \eqref{dw}
has been intensively studied.
In particular, if the damping is effective, we expect that
the critical exponent is the same as that of the corresponding parabolic problem.
Indeed, when
$\alpha = \beta = 0$,
it was shown by \cite{LiZh95, ToYo01, Zh01, KiQa02}
that the critical exponent is given by
$p_c = p_F(N) = 1+\frac{2}{N}$,
which is called the Fujita exponent named after the pioneering work by \cite{Fu66}.

For the time-dependent damping case, namely, $\alpha=0$,
it was revealed by \cite{LiNiZh12} that
the critical exponent remains $p_c = p_F(N)$ when $\alpha = 0, \beta \in (-1,1)$
(see \cite{Wa17, FuIkeWa, IkeSoWa} for the case $\beta = -1$).
When $\beta >1$, Lai and Takamura \cite{LaTa18} and Wakasa and Yordanov \cite{WakYo}
showed the small data blow-up
for the sub-Strauss or Strauss exponent
$1<p<\infty \ (N=1)$;
$1<p\le p_S(N) = \frac{N+1+\sqrt{N^2+10N-7}}{2(N-1)} \ (N\ge 2)$.
Recently, Liu and Wang \cite{LiWa} gave the global existence result
for $p>p_S(N)$ with $N=3,4$,
while the cases $N=2$ and $N\ge 5$ remain open.
On the other hand, in the scale-invariant case
$\beta = 1$, the situation becomes more complicated.
First, if $a_0$ is sufficiently large, we expect that the critical exponent coincides with $p_F(N)$.
Indeed, by \cite{DaLu13, Da15, Wa14}
it is proved that $p_c = p_F(N)$ holds for
$a_0 \ge \frac{5}{3} \ (N=1), 3 \ (N=2), N+2 \ (N\ge 3)$.
On the other hand, when
$a_0=2$,
D'Abbicco, Lucente and Reissig \cite{DaLuRe15} showed
$p_c = \max\{p_F(N), p_S(N+2) \}$
for $N\le 3$
(see \cite{DaLu15, Pa18a} for higher dimensional cases). 
This implies that the critical exponent depends on $a_0$.
For $a_0 \neq 2$, the best known result is by Ikeda and Sobajima \cite{IkeSopre3}.
They obtained the small data blow-up for
$N\ge 1$, $0<a_0 <\frac{N^2+N+2}{N+2}$ and $1<p\le p_S(N+a_0)$
(see also \cite{LaTaWa17, TuLi}).
Also, by \cite{IkeWa} it is proved that
the critical exponent is given by $p_c = 1$ when $\beta <-1$,
namely, the small data global existence holds for any $p>1$.

On the other hand, for the space-dependent damping case, namely, $\beta =0$,
Ikehata, Todorova and Yordanov \cite{IkToYo09} proved
if $0\le \alpha < 1$, then the critical exponent is given by
$p_c = 1+\frac{2}{N-\alpha}$.
However, there is no result in the case $\alpha < 0$,
where the damping coefficient is unbounded with respect to the space variable,
while we can expect that the critical exponent is still given by
$p_c = 1+\frac{2}{N-\alpha}$
in view of the result of linear problem \cite{SoWa18}.
When $\alpha =1$, the equation has scale-invariance and the critical exponent seems to change
depending on $a_0$.
Indeed, Ikeda and Sobajima \cite{IkeSopre2} proved the small data blow-up
for
$N\ge 3$, $0\le a_0 < \frac{(N-1)^2}{N+1}$, $\frac{N}{N-1}<p\le p_S(N+a_0)$,
while the global existence part remains open.

In contrast, there were only few results on time and space dependent cases.
By \cite{Wa12}, the small data global existence was proved
for $\alpha, \beta \ge 0, \alpha+\beta <1$ and $p>1+\frac{2}{N-\alpha}$.
Khader \cite{Kh13} also proved the small data global existence
for $0\le \alpha<1$, $-1<\beta<1$, $0<\alpha+\beta<1$, and
$p>1+\frac{4(\beta+1)}{2(N-\alpha)(\beta+1)-\beta(2-\alpha)}$.
However, there are no results on the small data blow-up for
subcritical or critical case.

We also refer the reader to
\cite{LiZh95, Ni03Ib, IkeWa15, LaZh, IkeIn, IkeSopre1,
Wak16}
for studies on estimates of lifespan of blow-up solutions.

Summarizing the previous studies, we can conjecture the following
for the critical exponent of the problem \eqref{dw} with
the condition \eqref{damping}.
\medskip

\noindent
{\bf Conjecture}\\
(i) For $\alpha < 2$, $\beta >-1$ with $\alpha+\beta < 1$,
the critical exponent is $p_c = 1+\frac{2}{N-\alpha}$.\\
(ii) For $\alpha, \beta \in \mathbb{R}$ with $\alpha + \beta =1$,
the equation has scale-invariance and the critical exponent will depend on $a_0$.\\
(iii) For $\alpha, \beta \in \mathbb{R}$ with $\alpha + \beta >1$,
the critical exponent is given by the Strauss number $p_c = p_S(N)$. 
\medskip

In this paper,
we present some partial answers which support the Conjecture (i).
In particular, we completely give the critical exponent in 
the case $\alpha<0, \beta=0$.

Before going to our main results, we mention the existence of the local solution.
\begin{proposition}[Existence of the local solution]\label{prop_lwp}
Let
$N \ge 1$
and let $c(t,x)$ has the form \eqref{damping}
with some $a_0>0$, $\alpha, \beta \in \mathbb{R}$.
We assume that the initial data satisfy \eqref{ini} with some $R_0>0$.
If $p$ satisfies
\begin{align}
\label{p}
	1 < p < \infty \ (N=1,2),\quad
	1 < p \le \frac{N}{N-2}\ (N \ge 3),
\end{align}
then for any
$\varepsilon >0$,
there exists a time
$T > 0$
such that the Cauchy problem \eqref{dw} admits a unique solution
\begin{align}
\label{u_class}
	u \in C([0, T) ; H^1(\mathbb{R}^N)) \cap C^1([0,T) ; L^2(\mathbb{R}^N))
\end{align}
satisfying
\begin{align}
\label{fps}
	{\rm supp\,}u(t,\cdot) \subset \{ x \in \mathbb{R}^N ; |x| \le R_0 + t \}.
\end{align}
Moreover, with the notion of the lifespan
\begin{align}
\label{ls}
	T(\varepsilon) :=
	\left\{ T \in (0,\infty] ;
	\ \mbox{there exists a unique solution} \ 
	u \ \mbox{in the class} \ \eqref{u_class} \right\},
\end{align}
we have the following blow-up alternative:
if
$T(\varepsilon) < \infty$,
then
\begin{align}
\label{bu_al}
	\lim_{t \to T(\varepsilon)-0} \| (u, u_t)(t) \|_{H^1\times L^2} = \infty
\end{align}
holds.
\end{proposition}
The proof is given by a standard energy estimate and the contraction mapping principle
(see \cite[Proposition 2.1]{IkTa05}).

Our first main results are the small data global existence
in the supercritical case.
\begin{theorem}[Small data global existence for $\beta \in (-1,1)$ in the supercritical case]\label{thm_gwp}
Let
$N \ge 1$
and let $c(t,x)$ has the form \eqref{damping}
with some $a_0>0$, $\alpha < 0$ and $\beta \in (-1,1)$.
We assume that the initial data satisfy \eqref{ini} with some $R_0>0$.
If $p$ satisfies
\begin{align}
\label{supercritical}
	1+\frac{2}{N-\alpha} < p < \infty \ (N=1,2),\quad
	1+\frac{2}{N-\alpha} < p \le \frac{N}{N-2}\ (N \ge 3),
\end{align}
then there exists a constant
$\varepsilon_0 > 0$
depending on
$N, p, a_0, \alpha, \beta, u_0, u_1, R_0$
such that for any
$\varepsilon \in (0, \varepsilon_0]$,
the Cauchy problem \eqref{dw} admits a unique global solution
\begin{align*}
	u \in C([0, \infty) ; H^1(\mathbb{R}^N)) \cap C^1([0,\infty) ; L^2(\mathbb{R}^N)).
\end{align*}
\end{theorem}

\begin{theorem}[Small data global existence for $\beta =1$ in the supercritical case]\label{thm_gwp_b1}
Let
$N \ge 1$
and let $c(t,x)$ has the form \eqref{damping}
with some $a_0 >0$, $\alpha < 0$ and $\beta = 1$.
We assume that the initial data satisfy \eqref{ini} with some $R_0>0$.
If $p$ satisfies \eqref{supercritical},
then there exist constants
$a_{\ast} > 0$
depending on $p$
and
$\varepsilon_0 > 0$
depending on
$N, p, a_0$, $\alpha, u_0, u_1, R_0$
such that if
$a_0 \ge a_{\ast}$
and
$\varepsilon \in (0, \varepsilon_0]$,
then
the Cauchy problem \eqref{dw} admits a unique global solution
\begin{align*}
	u \in C([0, \infty) ; H^1(\mathbb{R}^N)) \cap C^1([0,\infty) ; L^2(\mathbb{R}^N)).
\end{align*}
\end{theorem}

The second main result is
the finite time blow-up of the solution in the subcritical and the critical cases
when $\alpha$ and $\beta$ satisfy \eqref{beta}.
Moreover, we give the sharp upper estimates of the lifespan.
\begin{theorem}[Blow-up in the subcritical or the critical case]\label{thm_bu}
Let
$N \ge 1$,
and let $c(t,x)$ has the form \eqref{damping}
with some $a_0>0$, $\alpha$ and $\beta$ satisfying \eqref{beta}.
Moreover, we assume that $p$ satisfies
\begin{align}
\label{critical}
	1 < p \le 1+\frac{2}{N-\alpha},
\end{align}
and the initial data satisfy \eqref{ini} with some $R_0>0$ and
\begin{align}
\label{ini_pos}
	\int_{\mathbb{R}^N} \left(
		u_1(x) + (a_0\langle x \rangle^{-\alpha} - \beta) u_0(x)
		\right)\, dx > 0. 
\end{align}
Then, there exist constants
$\varepsilon_1 >0$
and
$C>0$
depending on
$N, p, a_0, \alpha, u_0, u_1, R_0$
such that for any
$\varepsilon \in (0, \varepsilon_1]$,
the lifespan of the local solution is estimated as
\begin{align}
\label{upp_ls}
	T(\varepsilon)
	\le
	\begin{cases}
		C \varepsilon^{-\frac{2-\alpha}{2(1+\beta)} \left( \frac{1}{p-1} - \frac{N-\alpha}{2} \right)^{-1}}
			&1<p<1+\frac{2}{N-\alpha},\\
		\exp \left( C \varepsilon^{-(p-1)} \right)
			&p = 1+ \frac{2}{N-\alpha}.
	\end{cases}
\end{align}
\end{theorem}

\begin{remark}
{\rm (i)} By Theorems \ref{thm_gwp}, and \ref{thm_bu} we conclude that
if the damping term is given by \eqref{damping} with some
$a_0 > 0$, $\alpha, \beta$ satisfying $\alpha <0, \beta =0$,
then the critical exponent of the Cauchy problem \eqref{dw} is determined by
\begin{align*}
	p_c = 1 + \frac{2}{N-\alpha}.
\end{align*}
When $\alpha < 0, \beta =1$,
by Theorems \ref{thm_gwp_b1} and \ref{thm_bu},
the critical exponent is almost determined as the above one,
in the sense that if
$1<p\le 1+\frac{2}{N-\alpha}$, then the small data blow-up occurs (for all $a_0 > 0$),
and if $p>1+\frac{2}{N-\alpha}$, then the small data global existence holds
provided that $a_0$ is sufficiently large (depending on $p$).

{\rm (ii)}
When the damping is effective, namely, $\alpha+\beta<1$,
if we do not impose the condition \eqref{beta},
the blow-up of solutions in the subcritical or the critical case
$1<p\le 1+\frac{2}{N-\alpha}$
is still an open problem.
\end{remark}

We shall comment on the method of proof and construction of the paper.
Theorem \ref{thm_gwp} is proved in Section 2 by weighted energy method
with a weight function having the form
$e^{\psi(t,x)}$ with an appropriate function $\psi(t,x)$ (see Definition \ref{2_def_psi}).
Such a weight function was developed by
\cite{Ik05IJPAM, ToYo09, Ni10, SoWa17}.
Making use of this weight, we can estimate the weighted energy of the solution
by the sum of the initial energy and the nonlinear terms (see Lemma \ref{2_lem_en}).
Then, to control the nonlinear terms, we apply
the Caffarelli--Kohn--Nirenberg inequality, which is suitable with
the energy including polynomially increasing coefficient $a(x)$.

Theorem \ref{thm_gwp_b1} can be also proved in the same strategy.
To avoid proceeding the similar computations as before,
we emphasize the difference from the proof of Theorem \ref{thm_gwp}
by giving only the outline of proof in Appendix B.

For Theorem \ref{thm_bu}, in Section 3, we apply the so-called test function method
developed by \cite{Zh01, LiNiZh12, DaLu13, IkeSopre1}.
Multiplying the equation by $(1+t)$, we can transform the linear part of the equation \eqref{dw} into
divergence form.
This is a simple but crucial idea in the proof.
Then, we further multiply a test function scaled by a large parameter $R \in (0,T(\varepsilon))$,
and apply the integration by parts, which gives a certain estimate
including the parameter $R$,
the initial data, and the nonlinear term.
Finally, letting $R$ to $T(\varepsilon)$, we have the estimate of the lifespan.


\section{Small data global existence in the supercritical case}
To keep the paper readable length,
we will give the detailed proof only for the case
$\alpha< 0$ and $\beta =0$, namely, $c(t,x) = a(x)$,
and for the other cases we will give an outline of the proof
in Appendix B.

\subsection{Construction of a weight function}
We first prepare a suitable weight function,
which will be used for weighted energy estimates of the solution.
In the previous works \cite{Ik05IJPAM, IkToYo09, Ni10, LiNiZh12, Wa12} for the case $\alpha \ge 0$,
the so-called Ikehata-Todorova-Yordanov type weight function
\begin{align}
\label{2_itypsi}
	e^{\psi(t,x)}\quad \mbox{with}\quad
		\psi(t,x) = \mu \frac{\langle x \rangle^{2-\alpha}}{1+t}, \ \mu>0
\end{align}
was used.
The function
$\psi (t,x)$ has the properties
\begin{align}%
\label{2_itypsi_1}
	-\psi_t(t,x) a(x) &= (2+\delta) |\nabla \psi(t,x)|^2,\\
\label{2_itypsi_2}
	\Delta \psi(t,x) &\ge \left( \frac{N-\alpha}{2(2-\alpha)} - \delta \right) \frac{a(x)}{1+t} 
\end{align}%
with some small $\delta >0$,
and these properties are essential for the weighted energy estimates.
However, if
$\alpha < 0$,
then the above weight function $\psi(t,x)$ does not satisfy the estimate \eqref{2_itypsi_2}.

Therefore, following the idea of \cite{SoWa17},
we modify the weight function \eqref{2_itypsi} as follows.
\begin{definition}\label{2_def_psi}
Let
$\delta \in (0,1)$,
$\mu = \frac{a_0}{(2-\alpha)^2(2+\delta)}$,
and
$A_0 > 0$.
Let
$R_{\delta}>0$ be a sufficiently large constant depending on $\delta > 0$, and
we take a cut-off function
$\eta_{R_{\delta}} \in C_0^{\infty}(\mathbb{R}^N)$
such that
\begin{align*}%
	\eta_{R_{\delta}}(x) = 1 \quad (|x| \le R_{\delta}),
	\quad \eta_{R_{\delta}}(x) = 0 \quad (|x| \ge 2R_{\delta}).
\end{align*}%
Let $\mathcal{N}$ be the Newton potential, that is,
\begin{align*}%
	\mathcal{N}(x) =
	\begin{cases}
	\displaystyle \frac{|x|}{2} &(N=1),\\[8pt]
	\displaystyle \frac{1}{2\pi} \log \frac{1}{|x|} &(N=2),\\[8pt]
	\displaystyle \frac{\Gamma(\frac{N}{2}+1)}{N(N-2) \pi^{N/2}} |x|^{2-N} &(N\ge 3).
	\end{cases}
\end{align*}%
We define
\begin{align*}%
	\psi(t,x) &= \frac{\mu}{1+t}
		\left\{ \langle x \rangle^{2-\alpha} + A_0
			- \mathcal{N}\ast \left(\alpha (2-\alpha) \langle x \rangle^{-2-\alpha} \eta_{R_{\delta}}(x) \right)
				\right\}.
\end{align*}%
\end{definition}
\begin{lemma}\label{2_lem_psi}
For any
$\delta \in (0,1)$,
there exist constants
$R_{\delta}>0$ and $A_0 > 0$ such that
the function $\psi(t,x)$ defined by Definition \ref{2_def_psi} satisfies
\begin{align}%
\label{2_sowapsi_1}
	-\psi_t(t,x) a(x) &\ge (2+\delta_1) |\nabla \psi(t,x)|^2,\\
\label{2_sowapsi_2}
	\Delta \psi(t,x) &\ge
		\left( \frac{N-\alpha}{2(2-\alpha)} - \delta_2 \right) \frac{a(x)}{1+t},
\end{align}%
where
$\delta_1 = \frac{2}{3}\delta$
and
$\delta_2 = \frac{N-\alpha}{2(2-\alpha)}\delta$. 
\end{lemma}
\begin{proof}
We calculate
\begin{align*}%
	\Delta \psi(t,x) &= 
		\mu (N-\alpha)(2-\alpha) \frac{\langle x \rangle^{-\alpha}}{1+t}
		+ \mu \alpha(2-\alpha) \frac{\langle x \rangle^{-2-\alpha}}{1+t} (1-\eta_{R_{\delta}}(x)) \\
	&\ge \left( \frac{N-\alpha}{2(2-\alpha)} - \delta_2 \right) \frac{a(x)}{1+t}
\end{align*}%
for sufficiently large
$R_{\delta}>0$.
Thus, we have \eqref{2_sowapsi_2}.
Next, we prove \eqref{2_sowapsi_1}.
We compute
\begin{align*}%
	&-\psi_t(t,x) a(x)\\
	&= \frac{1}{1+t} \psi(t,x) a(x) \\
	&= \frac{a_0}{(2-\alpha)^2 (2+\delta)} \frac{a(x)}{(1+t)^{2}}
		\left\{ \langle x \rangle^{2-\alpha} + A_0
			- \mathcal{N} \ast \left( \alpha(2-\alpha)\langle x \rangle^{-2-\alpha} \eta_{R_{\delta}}(x)
				 \right) \right\}.
\end{align*}%
On the other hand, we have
\begin{align*}%
	&(2+\delta_1) |\nabla \psi(t,x)|^2 \\
	&= (2+\delta_1) \frac{a_0^2}{(2-\alpha)^2(2+\delta)^2}
		\frac{\langle x \rangle^{-2\alpha}|x|^2}{(1+t)^{2}} \\
	&\quad - 2 (2+\delta_1)
		\frac{a_0^2}{(2-\alpha)^3(2+\delta)^2} \frac{\langle x \rangle^{-\alpha}x}{(1+t)^2}
			\cdot \nabla
			\left[ \mathcal{N} \ast
				\left( \alpha(2-\alpha) \langle x \rangle^{-2-\alpha} \eta_{R_{\delta}}(x) \right)
			\right]\\
	&\quad + (2+\delta_1) \frac{a_0^2}{(2-\alpha)^4(2+\delta)^2} \frac{1}{(1+t)^2}
			\left| \nabla \mathcal{N} \ast
				\left( \alpha(2-\alpha) \langle x \rangle^{-2-\alpha} \eta_{R_{\delta}}(x) \right)
			\right|^2.
\end{align*}%
We easily obtain
\begin{align*}%
	\left| \mathcal{N} \ast
				\left( \alpha(2-\alpha)  \langle x \rangle^{-2-\alpha} \eta_{R_{\delta}}(x) \right) \right|
		&\le C \begin{cases}
				\langle x \rangle^{2-N}&(N=1, N\ge 3),\\
				\log(2+|x|)&(N=2),
		\end{cases}\\
	\left| \nabla \mathcal{N} \ast
				\left( \alpha(2-\alpha) \langle x \rangle^{-2-\alpha} \eta_{R_{\delta}}(x) \right)
			\right|
		&\le C \langle x \rangle^{1-N}.
\end{align*}%
Now, we take a constant
$\delta_{3} > 0$
so that
$1-\delta_3 \ge \frac{2+\delta_1}{2+\delta}$
holds.
Then, we have
\begin{align*}%
	(1-\delta_3) \frac{a_0}{(2-\alpha)^2(2+\delta)}
		\frac{\langle x \rangle^{2-\alpha}a(x)}{(1+t)^2}
		\ge
		(2+\delta_1) \frac{a_0^2}{(2-\alpha)^2(2+\delta)^2}
			\frac{\langle x \rangle^{-2\alpha}|x|^2}{(1+t)^2}.
\end{align*}%
Finally, we take a sufficiently large constant $A_0 > 0$ so that
\begin{align*}%
	&\frac{a_0}{(2-\alpha)^2(2+\delta)}
		\frac{\delta_3 \langle x \rangle^{2-\alpha} a(x) + A_0 a(x)}{(1+t)^2} \\
	&\quad \ge
		\frac{a_0}{(2-\alpha)^2(2+\delta)} \frac{a(x)}{(1+t)^2}
			\left| \mathcal{N} \ast
				\left( \alpha(2-\alpha) \eta_{R_{\delta}}(x) \langle x \rangle^{-2-\alpha}\right)
			\right|\\
	&\quad +\left| 2(2+\delta_1) \frac{a_0^2}{(2-\alpha)^3(2+\delta)^2}
			\frac{\langle x \rangle^{-\alpha}x}{(1+t)^2}
				\cdot \nabla
				\left[ \mathcal{N} \ast
				\left( \alpha(2-\alpha) \eta_{R_{\delta}}(x) \langle x \rangle^{-2-\alpha}\right)
			\right] \right| \\
	&\qquad	+ (2+\delta_1) \frac{a_0^2}{(2-\alpha)^4(2+\delta)^2} \frac{1}{(1+t)^2}
			\left| \nabla \mathcal{N} \ast
				\left( \alpha(2-\alpha) \eta_{R_{\delta}}(x) \langle x \rangle^{-2-\alpha} \right)
			\right|^2.
\end{align*}%
Consequently, we have \eqref{2_sowapsi_1}.
\end{proof}

\subsection{Weighted energy estimates}
Using the function $\psi(t,x)$ constructed in Definition \ref{2_def_psi},
we prove the following weighted energy estimates for solutions to \eqref{dw}.
Continuing from the previous subsection, we consider the case
$\alpha < 0$ and $\beta = 0$.
Let
$\delta_0 \in (0,\frac{N-\alpha}{2-\alpha})$.
We define
\begin{align}%
\label{mt}
	M(t) &:= \sup_{0\le \tau <t}
		\left\{
		(1+\tau)^{\frac{N-\alpha}{2-\alpha} + 1 - \delta_0}
		\int_{\mathbb{R}^N} e^{2\psi} (u_t^2 + |\nabla u|^2 )\,dx \right.\\
\notag
		&\qquad \qquad \qquad\left.
		+(1+\tau)^{\frac{N-\alpha}{2-\alpha} - \delta_0}
		\int_{\mathbb{R}^N} e^{2\psi} a(x) u^2 \,dx
		\right\}.
\end{align}%

By the blow-up alternative in Proposition \ref{prop_lwp},
Theorem \ref{thm_gwp} is obtained from the following a priori estimate.

\begin{proposition}\label{2_prop_ap}
Under the assumptions of Theorem \ref{thm_gwp},
there exist
$\delta_0 \in (0, \frac{N-\alpha}{2-\alpha})$,
and constants
$t_0 = t_0(a_0, R_0, \delta_0) \ge 1$, $C = C(N,\alpha, a_0, p,R_0,\delta_0,t_0) >0$
such that
the solution $u$ constructed in Proposition \ref{prop_lwp} satisfies the a priori estimate
\begin{align*}%
	M(t) \le CM(0) + CM(t)^{(p+1)/2}
\end{align*}%
for $t \in [0,T(\varepsilon))$,
where the function
$\psi(t,x)$ is defined in Definition \ref{2_def_psi} with
$\delta = \frac{2-\alpha}{2(N-\alpha)}\delta_0$.
\end{proposition}

\begin{remark}\label{rem_p}
The restriction
$p \le \frac{N}{N-2} \ (N\ge 3)$
is due to the local existence (Proposition \ref{prop_lwp}),
and we can obtain the above a priori estimate for
$1+\frac{2}{N-\alpha} < p \le \frac{N+2}{N-2}$.
\end{remark}

To prove Proposition \ref{2_prop_ap},
we first prove the following energy estimate.
After that, in the next subsection,
we give the nonlinear estimates and complete the proof of Proposition \ref{2_prop_ap}.

\begin{lemma}\label{2_lem_en}
Under the assumptions of Proposition \ref{2_prop_ap},
for any
$\delta_0 \in (0, \frac{N-\alpha}{2-\alpha})$,
there exist constants
$t_0 = t_0(N, \alpha, a_0, R_0, \delta_0) \ge 1$, and $C = C(N,\alpha,a_0, p,R_0,\delta_0,t_0) >0$
such that
the solution $u$ constructed in Proposition \ref{prop_lwp} satisfies
\begin{align*}%
	&(t_0+t)^{\frac{N-\alpha}{2-\alpha}+1-\delta_0}
	\int_{\mathbb{R}^N} e^{2\psi} (u_t^2 + |\nabla u|^2 )\,dx
	+(t_0 + t)^{\frac{N-\alpha}{2-\alpha} - \delta_0}
	\int_{\mathbb{R}^N} e^{2\psi} a(x) u^2 \,dx\\
	&\le
	C \int_{\mathbb{R}^N} (u_1^2 + |\nabla u_0|^2 + u_0^2 + |u_0|^{p+1} )\,dx\\
	&\quad + C (t_0+t)^{\frac{N-\alpha}{2-\alpha}+1-\delta_0}
		\int_{\mathbb{R}^N} e^{2\psi} | F(u) | \,dx\\
	&\quad + C \int_0^t (t_0 + \tau)^{\frac{N-\alpha}{2-\alpha}-\delta_0}
		\int_{\mathbb{R}^N} e^{2\psi} |F(u)|\,dx d\tau\\
	&\quad + C \int_0^t (t_0 + \tau)^{\frac{N-\alpha}{2-\alpha}+1-\delta_0}
		\int_{\mathbb{R}^N} e^{2\psi}(-\psi_t) |F(u)| \,dx d\tau,
\end{align*}%
where
$F(u) = \int_0^u |v|^p \,dv$,
and
the function
$\psi(t,x)$ is defined in Definition \ref{2_def_psi} with
$\delta = \frac{2-\alpha}{2(N-\alpha)}\delta_0$.
\end{lemma}
\begin{proof}
Multiplying the equation of \eqref{dw} by $e^{2\psi}u_t$, we have
\begin{align*}%
	&\frac{\partial}{\partial t} \left[ \frac{e^{2\psi}}{2}( u_t^2 + |\nabla u|^2) \right]
	- \nabla\cdot \left( e^{2\psi} u_t \nabla u \right) \\
	&\quad + e^{2\psi} \left( a(x) - \frac{|\nabla \psi|^2}{-\psi_t} - \psi_t \right)u_t^2
		+ \frac{e^{2\psi}}{-\psi_t} | \psi_t \nabla u - u_t \nabla \psi |^2 \\
	&=  \frac{\partial}{\partial t} \left[ e^{2\psi}F(u) \right]
		+ 2 e^{2\psi} (-\psi_t) F(u).
\end{align*}%
By \eqref{2_sowapsi_1},
the last term of the left-hand side satisfies
\begin{align*}%
	\frac{e^{2\psi}}{-\psi_t} | \psi_t \nabla u - u_t \nabla \psi |^2
	\ge e^{2\psi} \left( \frac15 (-\psi_t)|\nabla u|^2 - \frac{a(x)}{4(2+\delta_1)} u_t^2 \right).
\end{align*}%
This and using again \eqref{2_sowapsi_1} imply
\begin{align}%
\label{2_en1}
	&\frac{\partial}{\partial t} \left[ \frac{e^{2\psi}}{2}( u_t^2 + |\nabla u|^2) \right]
	- \nabla\cdot \left( e^{2\psi} u_t \nabla u \right) \\
\notag
	&\quad + e^{2\psi} \left\{  \left( \frac{a(x)}{4} - \psi_t \right)u_t^2 + \frac{-\psi_t}{5} |\nabla u|^2 \right\}\\
\notag
	&\le  \frac{\partial}{\partial t} \left[ e^{2\psi}F(u) \right]
		+ 2 e^{2\psi} (-\psi_t) F(u).
\end{align}%

On the other hand, multiplying the equation of \eqref{dw} by
$e^{2\psi}u$, we have
\begin{align*}%
	 &\frac{\partial}{\partial t} \left[ e^{2\psi} \left(uu_t + \frac{a(x)}{2} u^2 \right) \right]
	 	- \nabla \cdot \left( e^{2\psi} u \nabla u \right)	 \\
	&\quad + e^{2\psi}
		\left\{
			|\nabla u|^2 -\psi_ta(x)u^2 + 2u \nabla\psi\cdot \nabla u - 2 \psi_t uu_t - u_t^2 \right\}\\
	&= e^{2\psi} |u|^pu.
\end{align*}%
By noting
\begin{align*}%
	2 e^{2\psi} u \nabla\psi \cdot \nabla u
	= 4 e^{2\psi} u \nabla\psi \cdot \nabla u
		- \nabla \cdot \left( e^{2\psi} u^2 \nabla \psi \right)
		+ 2 e^{2\psi} |\nabla \psi|^2 u^2 + e^{2\psi} (\Delta \psi) u^2,
\end{align*}%
we see that
\begin{align}%
\label{2_en2}
	&\frac{\partial}{\partial t} \left[ e^{2\psi} \left(uu_t + \frac{a(x)}{2} u^2 \right) \right]
	- \Delta \left( \frac{e^{2\psi}}{2} u^2 \right)\\
\notag
	&\quad
		+ e^{2\psi} \left\{ |\nabla u|^2 +4u\nabla \psi \cdot \nabla u
			+ \left( (-\psi_t)a(x) + 2|\nabla\psi|^2 \right) u^2 \right\} \\
\notag
	&\quad +e^{2\psi} \left( \Delta \psi \right) u^2
		- e^{2\psi} \left( 2\psi_t u u_t + u_t^2 \right)\\
\notag
	&= e^{2\psi} |u|^pu.
\end{align}%
By \eqref{2_sowapsi_1} and the Schwarz inequality, we estimate
\begin{align*}%
	&|\nabla u|^2 +4u\nabla \psi \cdot \nabla u
			+ \left( (-\psi_t)a(x) + 2|\nabla\psi|^2 \right) u^2 \\
	&\ge
		|\nabla u|^2 + 4u\nabla \psi \cdot \nabla u
			+ (4+\delta_4) |\nabla \psi|^2 u^2
			+ \frac{\delta}{4}(-\psi_t) a(x)u^2\\
	&\ge
		\delta_5 |\nabla u|^2 + \frac{\delta}{4}(-\psi_t) a(x)u^2,
\end{align*}%
where
$\delta_4 = \delta_1 - \frac{\delta}{2} - \frac{\delta \delta_1}{4}$
and
$\delta_5 = \frac{\delta_4}{4+\delta_4}$.
Here, we remark that
$\delta \in (0,1)$ and $\delta_1 = \frac23 \delta$ ensure $\delta_4 > 0$.
Also, \eqref{2_sowapsi_2} implies
\begin{align*}%
	e^{2\psi} (\Delta \psi) u^2
	\ge \left( \frac{N-\alpha}{2(2-\alpha)} - \delta_2 \right) e^{2\psi} \frac{a(x)}{1+t} u^2.
\end{align*}%
Moreover, it follows from the Schwarz inequality that
\begin{align*}%
	|2\psi_t u u_t| \le \frac{8}{\delta a(x)} (-\psi_t) u_t^2 + \frac{\delta}{8} a(x) (-\psi_t) u^2.
\end{align*}%
Plugging these into \eqref{2_en2}, we conclude
\begin{align*}%
	& \frac{\partial}{\partial t} \left[ e^{2\psi} \left( uu_t + \frac{a(x)}{2}u^2 \right) \right]
	 	-\Delta \left( \frac{e^{2\psi}}{2} u^2 \right)\\
	&\quad + e^{2\psi} \left( \delta_5 |\nabla u|^2 + \frac{\delta}{8}(-\psi_t) a(x)u^2 \right)
		+ \left( \frac{N-\alpha}{2(2-\alpha)} - \delta_2 \right) e^{2\psi} \frac{a(x)}{1+t} u^2 \\
	&\quad
		- e^{2\psi}\left( 1+ \frac{8(-\psi_t)}{\delta a(x)} \right) u_t^2\\
	&\le e^{2\psi} |u|^p u.
\end{align*}%
Integrating the above over $\mathbb{R}^N$ and
multiplying it by
$(t_0 + t)^{\frac{N-\alpha}{2-\alpha} - \delta_0}$,
we have
\begin{align}%
\label{2_en3}
	&\frac{d}{d t} \left[ (t_0 + t)^{\frac{N-\alpha}{2-\alpha} - \delta_0}
	 			\int_{\mathbb{R}^N} e^{2\psi} \left( uu_t + \frac{a(x)}{2}u^2 \right) \,dx \right] \\
\notag
	&\quad
		- \left( \frac{N-\alpha}{2-\alpha} - \delta_0 \right)
			(t_0 + t)^{\frac{N-\alpha}{2-\alpha} -1 - \delta_0}
			\int_{\mathbb{R}^N} e^{2\psi} \left( uu_t + \frac{a(x)}{2}u^2 \right) \,dx\\
\notag
	&\quad
		+ (t_0 + t)^{\frac{N-\alpha}{2-\alpha} - \delta_0}
			\int_{\mathbb{R}^N} e^{2\psi}
						\left( \delta_5 |\nabla u|^2 + \frac{\delta}{8}(-\psi_t) a(x)u^2 \right) \,dx\\
\notag
	&\quad
		+ \left( \frac{N-\alpha}{2(2-\alpha)} - \delta_2 \right)
			(t_0 + t)^{\frac{N-\alpha}{2-\alpha} -1 - \delta_0}
			\int_{\mathbb{R}^N} e^{2\psi} a(x) u^2 \,dx\\
\notag
	&\quad
		- (t_0 + t)^{\frac{N-\alpha}{2-\alpha} - \delta_0}
			\int_{\mathbb{R}^N} e^{2\psi}\left( 1+ \frac{8(-\psi_t)}{\delta a(x)} \right) u_t^2 \,dx\\
\notag
	&\le (t_0 + t)^{\frac{N-\alpha}{2-\alpha} - \delta_0}
		\int_{\mathbb{R}^N} e^{2\psi} |u|^p u\,dx.
\end{align}%
Since
$\delta = \frac{2-\alpha}{2(N-\alpha)}\delta_0$
and
$\delta_2 = \frac{N-\alpha}{2(2-\alpha)} \delta$,
we see that
$4 \delta_2 = \delta_0$ holds and hence,
we compute
\begin{align*}%
	&- \left( \frac{N-\alpha}{2-\alpha} - \delta_0 \right)
			(t_0 + t)^{\frac{N-\alpha}{2-\alpha} -1 - \delta_0}
			\int_{\mathbb{R}^N} e^{2\psi} \frac{a(x)}{2}u^2 \,dx \\
	&\quad +  \left( \frac{N-\alpha}{2(2-\alpha)} - \delta_2 \right)
			(t_0 + t)^{\frac{N-\alpha}{2-\alpha} -1 - \delta_0}
			\int_{\mathbb{R}^N} e^{2\psi} a(x) u^2 \,dx\\
	&= \delta_2 (t_0 + t)^{\frac{N-\alpha}{2-\alpha} -1 - \delta_0}
			\int_{\mathbb{R}^N} e^{2\psi} a(x) u^2 \,dx.
\end{align*}%
Moreover, the Schawarz inequality implies
\begin{align*}%
	&\left( \frac{N-\alpha}{2-\alpha} - \delta_0 \right)
			(t_0 + t)^{\frac{N-\alpha}{2-\alpha} -1 - \delta_0}
			\int_{\mathbb{R}^N} e^{2\psi} uu_t \,dx \\
	&\le \delta_6 (t_0 + t)^{\frac{N-\alpha}{2-\alpha} -1 - \delta_0}
			\int_{\mathbb{R}^N} e^{2\psi} a(x)u^2 \,dx
		+ C(\delta_6) (t_0 + t)^{\frac{N-\alpha}{2-\alpha} -1 - \delta_0}
			\int_{\mathbb{R}^N} e^{2\psi} \frac{1}{a(x)} u_t^2 \,dx\\
	&\le \delta_6 (t_0 + t)^{\frac{N-\alpha}{2-\alpha} -1 - \delta_0}
			\int_{\mathbb{R}^N} e^{2\psi} a(x) u^2 \,dx
		+ \frac{C(\delta_6)}{t_0} (t_0 + t)^{\frac{N-\alpha}{2-\alpha} - \delta_0}
			\int_{\mathbb{R}^N} e^{2\psi} \frac{1}{a(x)} u_t^2 \,dx
\end{align*}%
where
$\delta_6 = \frac{\delta_2}{2}$
and
$C(\delta_6) = \frac{1}{4\delta_6}(\frac{N-\alpha}{2-\alpha} - \delta_0)^2$.
Applying these estimates to \eqref{2_en3}, we deduce
\begin{align*}%
	&\frac{d}{d t} \left[ (t_0 + t)^{\frac{N-\alpha}{2-\alpha} - \delta_0}
	 			\int_{\mathbb{R}^N} e^{2\psi} \left( uu_t + \frac{a(x)}{2}u^2 \right) \,dx \right] \\
	&\quad
		+ (t_0 + t)^{\frac{N-\alpha}{2-\alpha} - \delta_0}
			\int_{\mathbb{R}^N} e^{2\psi}
						\left( \delta_5 |\nabla u|^2 + \frac{\delta}{8}(-\psi_t) a(x)u^2 \right) \,dx\\
\notag
	&\quad
		+ \delta_6 (t_0 + t)^{\frac{N-\alpha}{2-\alpha} -1 - \delta_0}
			\int_{\mathbb{R}^N} e^{2\psi} a(x) u^2 \,dx\\
\notag
	&\quad
		- (t_0 + t)^{\frac{N-\alpha}{2-\alpha} - \delta_0}
			\int_{\mathbb{R}^N} e^{2\psi}
					\left( 1+ \frac{C(\delta_6)}{t_0 a(x)}+ \frac{8(-\psi_t)}{\delta a(x)} \right) u_t^2 \,dx\\
\notag
	&\le (t_0 + t)^{\frac{N-\alpha}{2-\alpha} - \delta_0}
		\int_{\mathbb{R}^N} e^{2\psi} |u|^p u\,dx.
\end{align*}%
Integrating it over $[0,t]$, we conclude
\begin{align}%
\label{2_en4}
	&(t_0 + t)^{\frac{N-\alpha}{2-\alpha} - \delta_0}
	 			\int_{\mathbb{R}^N} e^{2\psi} \left( uu_t + \frac{a(x)}{2}u^2 \right) \,dx \\
\notag
	&\quad + \int_0^t (t_0+\tau)^{\frac{N-\alpha}{2-\alpha} - \delta_0}
					\int_{\mathbb{R}^N} e^{2\psi}
						\left( \delta_5 |\nabla u|^2 + \frac{\delta}{8}(-\psi_t) a(x)u^2 \right) \,dx d\tau\\
\notag
	&\quad + \delta_6 \int_0^t (t_0 + \tau )^{\frac{N-\alpha}{2-\alpha} -1 - \delta_0}
			\int_{\mathbb{R}^N} e^{2\psi} a(x) u^2 \,dx d\tau\\
\notag
	&\quad - \int_0^t (t_0 + \tau )^{\frac{N-\alpha}{2-\alpha} - \delta_0}
			\int_{\mathbb{R}^N} e^{2\psi}
				\left( 1+ \frac{C(\delta_6)}{t_0 a(x)}+ \frac{8(-\psi_t)}{\delta a(x)} \right) u_t^2
			\,dx d\tau\\
\notag
	&\le C(t_0) \int_{\mathbb{R}^N} e^{2\psi(0,x)} (u_0u_1 + \frac{a(x)}{2} u_0^2 )\,dx 
		+ \int_0^t (t_0 + \tau )^{\frac{N-\alpha}{2-\alpha} - \delta_0}
		\int_{\mathbb{R}^N} e^{2\psi} |u|^p u\,dx d\tau.
\end{align}%

On the other hand, integrating \eqref{2_en1} over $\mathbb{R}^N$, we have
\begin{align*}%
	&\frac{d}{d t} \int_{\mathbb{R}^N} \frac{e^{2\psi}}{2}( u_t^2 + |\nabla u|^2) \,dx \\
	&\quad + \int_{\mathbb{R}^N}
			e^{2\psi} \left\{  \left( \frac{a(x)}{4} - \psi_t \right)u_t^2 + \frac{-\psi_t}{5} |\nabla u|^2 \right\}\,dx \\
	&\le  \frac{d}{d t} \int_{\mathbb{R}^N} e^{2\psi}F(u) \, dx
		+ 2 \int_{\mathbb{R}^N} e^{2\psi} (-\psi_t) F(u) \,dx.
\end{align*}%
We multiply it by
$(t_0 + t)^{\frac{N-\alpha}{2-\alpha} + 1 - \delta_0}$
to obtain
\begin{align*}%
	&\frac{d}{d t} \left[ (t_0 + t)^{\frac{N-\alpha}{2-\alpha} + 1 - \delta_0}
		\int_{\mathbb{R}^N} \frac{e^{2\psi}}{2}( u_t^2 + |\nabla u|^2) \,dx \right] \\
	&\quad - \left\{ \frac{N-\alpha}{2-\alpha} + 1 -\delta_0 \right\}
		(t_0 + t)^{\frac{N-\alpha}{2-\alpha} - \delta_0}
		\int_{\mathbb{R}^N} \frac{e^{2\psi}}{2}( u_t^2 + |\nabla u|^2) \,dx \\
	&\quad + (t_0 + t)^{\frac{N-\alpha}{2-\alpha} + 1 - \delta_0} \int_{\mathbb{R}^N}
			e^{2\psi} \left\{  \left( \frac{a(x)}{4} - \psi_t \right)u_t^2 + \frac{-\psi_t}{5} |\nabla u|^2 \right\}\,dx \\
	&\le  \frac{d}{d t} \left[ (t_0 + t)^{\frac{N-\alpha}{2-\alpha} + 1 - \delta_0}
						\int_{\mathbb{R}^N} e^{2\psi}F(u) \, dx \right]\\
	&\quad - \left\{ \frac{N-\alpha}{2-\alpha} + 1 -\delta_0 \right\}
			(t_0 + t)^{\frac{N-\alpha}{2-\alpha} - \delta_0}
						\int_{\mathbb{R}^N} e^{2\psi}F(u) \, dx \\
	&\quad + 2 (t_0 + t)^{\frac{N-\alpha}{2-\alpha} + 1 - \delta_0} \int_{\mathbb{R}^N} e^{2\psi} (-\psi_t) F(u) \,dx.
\end{align*}%
Noting that
\begin{align*}%
	\left\{ \frac{N-\alpha}{2-\alpha} + 1 -\delta_0 \right\}
	(t_0 + t)^{\frac{N-\alpha}{2-\alpha} - \delta_0}
	\le (t_0 + t)^{\frac{N-\alpha}{2-\alpha} + 1 - \delta_0} \frac{a(x)}{8}
\end{align*}%
holds if we choose
$t_0\ge 1$ so that
$t_0 \ge \frac{8}{a_0}(\frac{N-\alpha}{2-\alpha} + 1 -\delta_0)$,
we calculate
\begin{align*}%
	&\frac{d}{d t} \left[ (t_0 + t)^{\frac{N-\alpha}{2-\alpha} + 1 - \delta_0}
		\int_{\mathbb{R}^N} \frac{e^{2\psi}}{2}( u_t^2 + |\nabla u|^2) \,dx \right] \\
	&\quad + (t_0 + t)^{\frac{N-\alpha}{2-\alpha} + 1 - \delta_0} \int_{\mathbb{R}^N}
			e^{2\psi} \left\{  \left( \frac{a(x)}{8} - \psi_t \right)u_t^2 + \frac{-\psi_t}{5} |\nabla u|^2 \right\}\,dx \\
	&\quad - \left\{ \frac{N-\alpha}{2-\alpha} + 1 -\delta_0 \right\}
		(t_0 + t)^{\frac{N-\alpha}{2-\alpha} - \delta_0}
		\int_{\mathbb{R}^N} \frac{e^{2\psi}}{2}|\nabla u|^2 \,dx \\
	&\le  \frac{d}{d t} \left[ (t_0 + t)^{\frac{N-\alpha}{2-\alpha} + 1 - \delta_0}
						\int_{\mathbb{R}^N} e^{2\psi}F(u) \, dx \right]\\
	&\quad - \left\{ \frac{N-\alpha}{2-\alpha} + 1 -\delta_0 \right\}
			(t_0 + t)^{\frac{N-\alpha}{2-\alpha} - \delta_0}
						\int_{\mathbb{R}^N} e^{2\psi}F(u) \, dx \\
	&\quad + 2 (t_0 + t)^{\frac{N-\alpha}{2-\alpha} + 1 - \delta_0} \int_{\mathbb{R}^N} e^{2\psi} (-\psi_t) F(u) \,dx.
\end{align*}%
Integrating it over $[0,t]$, we conclude
\begin{align}%
\label{2_en5}
	&(t_0 + t)^{\frac{N-\alpha}{2-\alpha} + 1 - \delta_0}
		\int_{\mathbb{R}^N} \frac{e^{2\psi}}{2}( u_t^2 + |\nabla u|^2) \,dx \\
\notag
	&\quad + \int_0^t (t_0 + \tau)^{\frac{N-\alpha}{2-\alpha} + 1 - \delta_0} \int_{\mathbb{R}^N}
			e^{2\psi} \left\{  \left( \frac{a(x)}{8} - \psi_t \right)u_t^2 + \frac{-\psi_t}{5} |\nabla u|^2 \right\}\,dx d\tau \\
\notag
	&\quad - \left\{ \frac{N-\alpha}{2-\alpha} + 1 -\delta_0 \right\}
		\int_0^t (t_0 + \tau)^{\frac{N-\alpha}{2-\alpha} - \delta_0}
		\int_{\mathbb{R}^N} \frac{e^{2\psi}}{2}|\nabla u|^2 \,dx d\tau \\
\notag
	&\le C(t_0) \int_{\mathbb{R}^N} e^{2\psi(0,x)} (u_1^2 + |\nabla u_0|^2 + |u_0|^{p+1} )\,dx\\
\notag
	&\quad + (t_0 + t)^{\frac{N-\alpha}{2-\alpha} + 1 - \delta_0}
						\int_{\mathbb{R}^N} e^{2\psi}F(u) \, dx \\
\notag
	&\quad - \left\{ \frac{N-\alpha}{2-\alpha} + 1 -\delta_0 \right\}
			\int_0^t (t_0 + \tau)^{\frac{N-\alpha}{2-\alpha} - \delta_0}
						\int_{\mathbb{R}^N} e^{2\psi}F(u) \, dx d\tau\\
\notag
	&\quad + 2 \int_0^t (t_0 + \tau )^{\frac{N-\alpha}{2-\alpha} + 1 - \delta_0}
					\int_{\mathbb{R}^N} e^{2\psi} (-\psi_t) F(u) \,dx d\tau.
\end{align}%
Finally, we combine \eqref{2_en4} and \eqref{2_en5}.
We take $\nu > 0$ so that
\begin{align*}%
	\nu \left\{ \frac{N-\alpha}{2-\alpha} + 1 -\delta_0 \right\} \le \frac{\delta_5}{2}
\end{align*}%
holds and, we choose $t_0$ sufficiently large so that
\begin{align*}%
	\frac{1}{t_0} \left( 1+\frac{C(\delta_6)}{t_0} \right) &\le \frac{\nu}{16}a(x),\\
	\frac{8}{\delta t_0 a(x)} \le \frac{\nu}{2},
\end{align*}%
Computing \eqref{2_en4} $+\nu\cdot$\eqref{2_en5}, we conclude
\begin{align*}%
	&\nu (t_0 + t)^{\frac{N-\alpha}{2-\alpha} + 1 - \delta_0}
		\int_{\mathbb{R}^N} \frac{e^{2\psi}}{2}( u_t^2 + |\nabla u|^2) \,dx
	+(t_0 + t)^{\frac{N-\alpha}{2-\alpha} - \delta_0}
	 			\int_{\mathbb{R}^N} e^{2\psi} \left( uu_t + \frac{a(x)}{2}u^2 \right) \,dx \\
	&\quad
		+ \nu \int_0^t (t_0 + \tau)^{\frac{N-\alpha}{2-\alpha} + 1 - \delta_0} \int_{\mathbb{R}^N}
			e^{2\psi} \left\{  \left( \frac{a(x)}{16} + \frac{-\psi_t}{2} \right)u_t^2
						+ \frac{-\psi_t}{5} |\nabla u|^2 \right\}\,dx d\tau\\
	&\quad + \int_0^t (t_0+\tau)^{\frac{N-\alpha}{2-\alpha} - \delta_0}
					\int_{\mathbb{R}^N} e^{2\psi}
						\left( \frac{\delta_5}{2} |\nabla u|^2 + \frac{\delta}{8}(-\psi_t) a(x)u^2 \right) \,dx d\tau\\
	&\quad + \delta_6 \int_0^t (t_0 + \tau )^{\frac{N-\alpha}{2-\alpha} -1 - \delta_0}
			\int_{\mathbb{R}^N} e^{2\psi} a(x) u^2 \,dx d\tau\\
	&\le C(t_0) \int_{\mathbb{R}^N} e^{2\psi(0,x)} (u_1^2 +|\nabla u_0|^2 +|u_0|^{p+1}+ a(x)u_0^2 )\,dx \\
	&\quad + \nu (t_0 + t)^{\frac{N-\alpha}{2-\alpha} + 1 - \delta_0}
						\int_{\mathbb{R}^N} e^{2\psi} |F(u)| \, dx\\
	&\quad + C(N, \alpha, \delta_0, \nu) \int_0^t (t_0 + \tau )^{\frac{N-\alpha}{2-\alpha} - \delta_0}
				\int_{\mathbb{R}^N} e^{2\psi} |F(u)| \,dx d\tau \\
	&\quad + C(\nu) \int_0^t (t_0 + \tau )^{\frac{N-\alpha}{2-\alpha} + 1 - \delta_0}
					\int_{\mathbb{R}^N} e^{2\psi} (-\psi_t) |F(u)| \,dx d\tau.
\end{align*}%
Finally, noting that
\begin{align*}%
	\nu (t_0 + t)^{\frac{N-\alpha}{2-\alpha} + 1 - \delta_0}
		\int_{\mathbb{R}^N} \frac{e^{2\psi}}{2}( u_t^2 + |\nabla u|^2) \,dx
	+(t_0 + t)^{\frac{N-\alpha}{2-\alpha} - \delta_0}
	 			\int_{\mathbb{R}^N} e^{2\psi} \left( uu_t + \frac{a(x)}{2}u^2 \right) \,dx \\
	\ge
	\frac{\nu}{2} (t_0 + t)^{\frac{N-\alpha}{2-\alpha} + 1 - \delta_0}
		\int_{\mathbb{R}^N} \frac{e^{2\psi}}{2}( u_t^2 + |\nabla u|^2) \,dx
	+ \frac14 (t_0 + t)^{\frac{N-\alpha}{2-\alpha} - \delta_0}
	 			\int_{\mathbb{R}^N} e^{2\psi} a(x) u^2 \,dx
\end{align*}%
if $t_0 \ge \frac{2}{\nu a_0}$,
we have the assertion of Lemma \ref{2_lem_en}.
\end{proof}

\begin{remark}\label{rem_parameters}
In the above proof, we have determined
the positive constants
$\delta, \delta_j \ (j = 1,\ldots,6), \nu, t_0$
in the following way.
First, for given
$\delta_0 \in (0,\frac{N-\alpha}{2-\alpha})$,
we define
$\delta := \frac{2-\alpha}{2(N-\alpha)}\delta_0$.
Then, we choose $\delta_j \ (j =1,\ldots,6)$ as
$\delta_1 := \frac23 \delta$,
$\delta_2 := \frac{N-\alpha}{2(2-\alpha)}\delta$,
$\delta_3$ such that $1-\delta_3 \ge \frac{2+\delta_1}{2+\delta}$,
$\delta_4:= \delta_1 - \frac{\delta}{2} - \frac{\delta \delta_1}{4}$,
$\delta_5 := \frac{\delta_4}{4+\delta_4}$,
$\delta_6 := \frac{\delta_2}{2}$.
Note that
$\delta = O(\delta_0)$
amd
$\delta_j = O(\delta_0) \ (j=1,\ldots,6)$
as $\delta_0 \to 0+$.
After that, we take $\nu$ as
$\nu (\frac{N-\alpha}{2-\alpha} + 1 - \delta_0) \le \frac{\delta_5}{2}$.
Then, we choose $t_0 \ge 1$ so that
\begin{align*}%
	&t_0 \ge \frac{8}{a_0}(\frac{N-\alpha}{2-\alpha} + 1 -\delta_0),
	\quad
	\frac{1}{t_0} \left( 1+\frac{C(\delta_6)}{t_0} \right) \le \frac{\nu}{16}a(x),\\
	&\frac{8}{\delta t_0 a(x)} \le \frac{\nu}{2},
	\quad
	t_0 \ge \frac{2}{\nu a_0}
\end{align*}%
hold for any $x \in \mathbb{R}^N$.
These observations will be useful when we discuss the case
$\alpha < 0, \beta = 1$
(see Appendix B).

Finally, we also remark that
$\delta_0$ will be determined depending on $p$ in the nonlinear estimates discussed in the next subsection.
\end{remark}

\subsection{Nonlinear estimates and proof of Proposition \ref{2_prop_ap}}
In this subsection,
we give the nonlinear estimates for the right-hand side of Lemma \ref{2_lem_en}
and complete the proof of Proposition \ref{2_prop_ap}.
We first recall the following special case of Gagliardo--Nirenberg inequality.
\begin{lemma}\label{GN_ineq}
{\rm (Gagliardo--Nirenberg inequality, see \cite[Section 6.1.1]{GiGiSa})}
For $1\le p < \infty \ (N=1,2)$, $1 \le p \le \frac{N+2}{N-2} \ (N\ge 3)$, there exists a constant $C>0$ such that
for any $u \in H^1(\mathbb{R}^N)$, we have
\begin{align*}%
	\| u \|_{L^{p+1}} \le C \| \nabla u \|_{L^2}^{\theta} \| u \|_{L^2}^{1-\theta},
\end{align*}%
where $\theta = \frac{N(p-1)}{2(p+1)} \in [0,1]$.
\end{lemma}

Besides the above lemma, we also use the following
special case of Caffarelli--Kohn--Nirenberg inequality.
\begin{lemma}\label{CKN_ineq}
{\rm (Caffarelli--Kohn--Nirenberg inequality, see \cite{CaKoNi84})}
For $k= 0,1,2,\ldots$, there exist $C(k,N) > 0$ such that for $u \in H^1(\mathbb{R}^N)$
with compact support, we have
\begin{align*}%
	\| u \|_{L^2} \le C(k,N) \| \nabla u \|_{L^2}^{1-1/2^k} \| |x|^{2^k-1} u \|_{L^2}^{1/2^k}.
\end{align*}%
\end{lemma}
We give a short proof of this lemma in Appendix.

Based on the above lemma, we first prepare the following.
\begin{lemma}\label{2_lem_int}
Under the assumptions of Proposition \ref{2_prop_ap},
for any integer $k$ satisfying $2^k-1 \ge - \frac{\alpha}{2}$,
there exists a constant
$C(\alpha,p,k,N) > 0$ such that
for any $u \in H^1(\mathbb{R}^N)$ with compact support, we have
\begin{align*}%
	\| e^{\frac{2\psi}{p+1}} u \|_{L^2}
	&\le C(\alpha, p, k, N) 
		\left( (1+t)^{-\frac12} \| e^{\psi} \sqrt{a} u \|_{L^2} + \| e^{\psi} \nabla u\|_{L^2} \right)^{1-1/2^k} \\
	&\quad \times
		\left( (1+t)^{(2^k-1+\frac{\alpha}{2})/(2-\alpha)} \| e^{\psi} \sqrt{a} u \|_{L^2} \right)^{1/2^k}.
\end{align*}%
\end{lemma}
\begin{proof}
By Lemma \ref{CKN_ineq}, we have
\begin{align}%
\label{2_e_psi1}
	\| e^{\frac{2\psi}{p+1}} u \|_{L^2}
	&\le C(k,N) \| \nabla (e^{\frac{2\psi}{p+1}} u ) \|_{L^2}^{1-1/2^k}
			\| \langle x \rangle^{2^k-1} e^{\frac{2\psi}{p+1}} u \|_{L^2}^{1/2^k}
\end{align}%
with $k$ satisfying $2^k-1 \ge -\frac{\alpha}{2}$.
We estimate
\begin{align*}%
	\langle x \rangle^{2^k-1} e^{\frac{2\psi}{p+1}}
		&= \langle x \rangle^{-\frac{\alpha}{2}} \langle x \rangle^{2^k-1 + \frac{\alpha}{2}} e^{\frac{2\psi}{p+1}}\\
		&= \langle x \rangle^{-\frac{\alpha}{2}}
			\left( \frac{\langle x \rangle^{2-\alpha}}{1+t} \right)^{(2^k-1 + \frac{\alpha}{2})/(2-\alpha)}
				e^{\frac{2\psi}{p+1}} \cdot (1+t)^{(2^k-1 + \frac{\alpha}{2})/(2-\alpha)} \\
		&\le C(\alpha,p,k,N) (1+t)^{(2^k-1 + \frac{\alpha}{2})/(2-\alpha)}
			e^{\psi(t,x)} \sqrt{a(x)}.
\end{align*}%
On the other hand, we compute
\begin{align*}%
	\nabla ( e^{\frac{2\psi}{p+1}} u ) 
		= e^{\frac{2\psi}{p+1}}
			\left( \frac{2}{p+1} (\nabla \psi) u + \nabla u \right)
\end{align*}%
and
\begin{align*}%
	| \nabla \psi | e^{\frac{2\psi}{p+1}}
	&\le C \frac{\langle x \rangle^{1-\alpha}}{1+t} e^{\frac{2\psi}{p+1}} \\
	&\le C \langle x \rangle^{-\frac{\alpha}{2}}
			\left( \frac{\langle x \rangle^{2-\alpha}}{1+t} \right)^{(1-\frac{\alpha}{2})/(2-\alpha)}
				e^{\frac{2\psi}{p+1}} \cdot (1+t)^{(1-\frac{\alpha}{2})/(2-\alpha) -1} \\
	&\le C(1+t)^{-\frac12} e^{\psi(t,x)} \sqrt{a(x)}.
\end{align*}%
Plugging these estimates into \eqref{2_e_psi1}, we have the desired estimate.
\end{proof}

Combining Lemmas \ref{GN_ineq} and \ref{2_lem_int}, we obtain the following interpolation estimate.
\begin{lemma}\label{2_lem_int2}
Under the assumptions of Proposition \ref{2_prop_ap},
for any integer $k$ satisfying $2^k-1 \ge - \frac{\alpha}{2}$,
there exists a constant
$C(\alpha,p,k,N) > 0$ such that
for any $u \in H^1(\mathbb{R}^N)$ with compact support, we have
\begin{align*}%
	\| e^{\frac{2\psi}{p+1}} u \|_{L^{p+1}}
	&\le C(\alpha, p, k, N) 
		\left( (1+t)^{-\frac12} \| e^{\psi} \sqrt{a} u \|_{L^2} + \| e^{\psi} \nabla u\|_{L^2}
				\right)^{\theta + (1-1/2^k)(1-\theta)} \\
	&\quad \times
		\left( (1+t)^{(2^k-1+\frac{\alpha}{2})/(2-\alpha)} \| e^{\psi} \sqrt{a} u \|_{L^2} \right)^{(1-\theta)/2^k},
\end{align*}%
where $\theta = \frac{N(p-1)}{2(p+1)} \in [0,1]$.
\end{lemma}
\begin{proof}
By Lemmas \ref{GN_ineq} and \ref{2_lem_int}, we estimate
\begin{align*}%
	\| e^{\frac{2\psi}{p+1}} u \|_{L^{p+1}}
	&\le C \| \nabla (e^{\frac{2\psi}{p+1}} u) \|_{L^2}^{\theta} \| e^{\frac{2\psi}{p+1}} u \|_{L^2}^{1-\theta} \\
	&\le C \left( (1+t)^{-\frac12} \| e^{\psi}\sqrt{a}u \|_{L^2} + \| e^{\psi} \nabla u \|_{L^2} \right)^{\theta}
			\| e^{\frac{2\psi}{p+1}} u \|_{L^2}^{1-\theta} \\
	&\le C \left( (1+t)^{-\frac12} \| e^{\psi}\sqrt{a}u \|_{L^2} + \| e^{\psi} \nabla u \|_{L^2} \right)^{\theta} \\
	&\quad \times
		\left( (1+t)^{-\frac12} \| e^{\psi} \sqrt{a} u \|_{L^2} + \| e^{\psi} \nabla u\|_{L^2} \right)^{(1-1/2^k)(1-\theta)} \\
	&\quad \times
		\left( (1+t)^{(2^k-1+\frac{\alpha}{2})/(2-\alpha)} \| e^{\psi} \sqrt{a} u \|_{L^2} \right)^{(1/2^k)(1-\theta)}\\
	&= C \left( (1+t)^{-\frac12} \| e^{\psi} \sqrt{a} u \|_{L^2} + \| e^{\psi} \nabla u\|_{L^2}
				\right)^{\theta + (1-1/2^k)(1-\theta)} \\
	&\quad \times
		\left( (1+t)^{(2^k-1+\frac{\alpha}{2})/(2-\alpha)} \| e^{\psi} \sqrt{a} u \|_{L^2} \right)^{(1-\theta)/2^k},
\end{align*}%
which completes the proof.
\end{proof}

Now we are in a position to estimate the nonlinearities
\begin{align*}%
	&C (t_0+t)^{\frac{N-\alpha}{2-\alpha}+1-\delta_0}
		\int_{\mathbb{R}^N} e^{2\psi} | F(u) | \,dx\\
	&\quad + C \int_0^t (t_0 + \tau)^{\frac{N-\alpha}{2-\alpha}-\delta_0}
		\int_{\mathbb{R}^N} e^{2\psi} |F(u)|\,dx d\tau\\
	&\quad + C \int_0^t (t_0 + \tau)^{\frac{N-\alpha}{2-\alpha}+1-\delta_0}
		\int_{\mathbb{R}^N} e^{2\psi}(-\psi_t) |F(u)| \,dx d\tau \\
	&=: N_1 + N_2 + N_3
\end{align*}%
in the right-hand side of Lemma \ref{2_lem_en}.
We first consider $N_1$.
Applying Lemma \ref{2_lem_int2} and using the definition of $M(t)$ (see \eqref{mt}), we deduce
\begin{align*}%
	N_1 &\le 
		C (t_0+t)^{\frac{N-\alpha}{2-\alpha}+1-\delta_0} \| e^{\frac{2\psi}{p+1}}u \|_{L^{p+1}}^{p+1}\\
		&\le C (t_0+t)^{\frac{N-\alpha}{2-\alpha}+1-\delta_0}
			\left( (1+t)^{-\frac12} \| e^{\psi} \sqrt{a} u \|_{L^2} + \| e^{\psi} \nabla u\|_{L^2}
				\right)^{(\theta + (1-1/2^k)(1-\theta))(p+1)} \\
	&\quad \times
		\left( (1+t)^{(2^k-1+\frac{\alpha}{2})/(2-\alpha)} \| e^{\psi} \sqrt{a} u \|_{L^2} \right)^{(1-\theta)(p+1)/2^k} \\
	&\le C (t_0+t)^{\frac{N-\alpha}{2-\alpha}+1-\delta_0}
		(1+t)^{-\frac{1}{2}(\frac{N-\alpha}{2-\alpha}+1-\delta_0)(\theta + (1-1/2^k)(1-\theta))(p+1)}\\
	&\quad \times
		(1+t)^{[(2^k-1+\frac{\alpha}{2})/(2-\alpha) - \frac{1}{2}(\frac{N-\alpha}{2-\alpha}-\delta_0)](1-\theta)(p+1)/2^k}
		M(t)^{\frac{p+1}{2}}.
\end{align*}%
By a straightforward calculation, we can see that the condition
\begin{align*}%
	&\frac{N-\alpha}{2-\alpha}+1-\delta_0
	-\frac{1}{2} \left( \frac{N-\alpha}{2-\alpha}+1-\delta_0 \right) (\theta + \left( 1-\frac{1}{2^k} \right)(1-\theta))(p+1) \\
	&+\left[ \left( 2^k-1+\frac{\alpha}{2} \right)\frac{1}{2-\alpha} 
				- \frac{1}{2}(\frac{N-\alpha}{2-\alpha}-\delta_0) \right]
		(1-\theta)(p+1) \frac{1}{2^k}
	<0
\end{align*}%
if and only if
\begin{align}%
\label{2_p}
	p > 1+ \frac{2}{N-\alpha - (2-\alpha)\delta_0/2}.
\end{align}%
Noting $p > 1+ \frac{2}{N-\alpha}$ and
$\delta_0 < \frac{2}{2-\alpha}(N-\alpha-\frac{2}{p-1})$,
we see have \eqref{2_p} holds and hence,
\begin{align}%
\label{2_n1}
	N_1 \le C M(t)^{\frac{p+1}{2}}.
\end{align}%
We can obtain the same estimate as \eqref{2_n1} for $N_2$ in the same way.
Finally, for $N_3$, noting
\begin{align*}%
	(-\psi_t) e^{2\psi} &\le C \frac{\langle x \rangle^{2-\alpha}}{(1+t)^2} e^{2\psi} 
	\le C (1+t)^{-1} e^{\frac{p+3}{2}\psi},
\end{align*}%
we have
\begin{align*}%
	N_3 \le C \int_0^t (t_0+\tau)^{\frac{N-\alpha}{2-\alpha}-\delta_0}
			\int_{\mathbb{R}^N} e^{\frac{p+3}{2}\psi} |F(u)| \,dx d\tau.
\end{align*}%
We can apply the same argument to the right-hand side and obtain the
same estimate as \eqref{2_n1} for $N_3$.
This completes the proof of Proposition \ref{2_prop_ap}.

\section{Blow-up and the sharp upper estimates of the lifespan
in the subcritical and the critical case}
In this section, we give a proof of Theorem \ref{thm_bu}
for the case
$\alpha < 0, \beta =1$.
We can also prove Theorem \ref{thm_bu} for the case $\alpha < 0, \beta =0$
in the same argument with a slight modification
(see Remark \ref{rem_beta0} below).
The proof is based on the test-function method developed by
Ikeda and Sobajima \cite{IkeSopre1}.

First, we remark that
if $T(\varepsilon) \le R_0$, then the assertion of Theorem \ref{thm_bu} is obvious,
provided that
$\varepsilon \le 1$.
Thus, we may assume that $T(\varepsilon) > R_0$.
Let $\eta=\eta(s)$ be a test function such that
\begin{align*}%
	\eta(s) = \begin{cases}
		1 &\mbox{if}\ s\le \frac12,\\
		\mbox{decreasing} &\mbox{if} \ \frac12 < s < 1,\\
		0&\mbox{if} \ s \ge 1.
	\end{cases}
\end{align*}%
Let
$R \in [R_0,T(\varepsilon))$
a parameter and,
we define
\begin{align*}%
	\psi_R(t,x) = \left[ \eta\left( \frac{|x|^{2-\alpha}+t^2}{R^2} \right) \right]^{2p'}.
\end{align*}%
We also define
\begin{align*}%
	\eta^{\ast}(s) = \begin{cases}
		0&\mbox{if} \ s \le \frac12,\\
		\eta(s)&\mbox{if} \ s>\frac12,
	\end{cases}
	\quad
	\psi_R^{\ast}(t,x) = \left[ \eta^{\ast}\left( \frac{|x|^{2-\alpha}+t^2}{R^2} \right) \right]^{2p'}.
\end{align*}%
Then, we have
\begin{align*}%
	|\partial_t \psi_R(t,,x)| &\le CR^{-2} (1+t) \left[ \psi_R^{\ast}(t,x) \right]^{1-\frac{1}{2p'}},\\
	|\partial_t^2 \psi_R(t,x)| &\le CR^{-2} \left[ \psi_R^{\ast}(t,x) \right]^{\frac{1}{p}},\\
	|\Delta \psi_R(t,x)| &\le C R^{-2} \langle x \rangle^{-\alpha} \left[ \psi_R^{\ast}(t,x) \right]^{\frac{1}{p}}
\end{align*}%
and
\begin{align*}%
	\langle x \rangle^{-\alpha} \le C R^{-\frac{2\alpha}{2-\alpha}} \quad
	\mbox{on}\ \mbox{supp}\, \psi_R, \ \mbox{for}\ R\ge R_0.
\end{align*}%
Finally, we define
\begin{align*}%
	\Psi_R(t,x) = (1+t) \psi_R(t,x).
\end{align*}%
Multiplying the equation \eqref{dw} by
$\Psi_R$ and integrating it over $\mathbb{R}^N$, we have
\begin{align*}%
	\int_{\mathbb{R}^N} |u|^p \Psi_R\, dx
	&= \int_{\mathbb{R}^N} \left(
			u_{tt} - \Delta u + \frac{a(x)}{1+t} u_t \right) \Psi_R \,dx.
\end{align*}%
By integration by parts, we calculate
\begin{align*}%
	\int_{\mathbb{R}^N} |u|^p \Psi_R\, dx
	&= \frac{d}{dt}
		\int_{\mathbb{R}^N} \left( u_t \Psi_R - u\partial_t \Psi_R + \frac{a(x)}{1+t} u \Psi_R \right) \,dx \\
	&\quad + \int_{\mathbb{R}^N} u
			\left( \partial_t^2\Psi_R - \Delta \Psi_R - \partial_t \left( \frac{a(x)}{1+t}\Psi_R \right) \right)
			\,dx \\
	&= \frac{d}{dt} \int_{\mathbb{R}^N} \left( u_t \Psi_R - u\partial_t \Psi_R + \frac{a(x)}{1+t} u \Psi_R \right) \,dx \\
	&\quad + \int_{\mathbb{R}^N}
		u \left( 2 \partial_t \psi_R + (1+t) \partial_t^2\psi_R + (1+t) \Delta \psi_R- a(x) \partial_t \psi_R \right) \,dx \\
	&\le \frac{d}{dt} \int_{\mathbb{R}^N} \left( u_t \Psi_R - u\partial_t \Psi_R + \frac{a(x)}{1+t} u \Psi_R \right) \,dx \\
	&\quad + CR^{-2} \int_{\mathbb{R}^N}
				|u| \left( 1+\langle x \rangle^{-\alpha} \right) (1+t) [\psi_R^{\ast} ]^{\frac{1}{p}}
				\,dx \\
	&\le \frac{d}{dt} \int_{\mathbb{R}^N} \left( u_t \Psi_R - u\partial_t \Psi_R + \frac{a(x)}{1+t} u \Psi_R \right) \,dx \\
	&\quad + CR^{-\frac{4}{2-\alpha}}
			\int_{\mathbb{R}^N} |u| (1+t) [\psi_R^{\ast}]^{\frac{1}{p}} \,dx.
\end{align*}%
Integrating it over $[0,R]$ and applying the H\"{o}lder inequality, we have
\begin{align}%
\label{3_ineq}
	&\varepsilon \int_{\mathbb{R}^N} (u_1(x) + (a(x)-1) u_0(x))\,dx
		+ \iint_{\mathbb{R}^N\times (0,R)} |u|^p (1+t) \psi_R\,dxdt \\
\notag
	&\le
		CR^{-\frac{4}{2-\alpha}}
		\left( \iint_{\mathbb{R}^N\times (0,R)} |u|^p (1+t) \psi_R^{\ast} \,dxdt\right)^{\frac{1}{p}}
		\left( \iint_{|x|^{2-\alpha}+t^2\le R^2} (1+t) \,dxdt \right)^{\frac{1}{p'}} \\
\notag
	&\le CR^{-\frac{4}{2-\alpha} (\frac{1}{p-1} - \frac{N-\alpha}{2})\frac{1}{p'}}
		\left( \iint_{\mathbb{R}^N\times (0,R)} |u|^p (1+t) \psi_R^{\ast} \,dxdt\right)^{\frac{1}{p}}.
\end{align}%

In the subcritical case
$1<p<1+\frac{2}{N-\alpha}$,
from $\psi_R^{\ast} \le \psi_R$ and the Young inequality, we obtain
\begin{align*}%
	\varepsilon \int_{\mathbb{R}^N} (u_1(x) + (a(x)-1) u_0(x))\,dx
	&\le CR^{-\frac{4}{2-\alpha} (\frac{1}{p-1} - \frac{N-\alpha}{2})}.
\end{align*}%
By the assumption of Theorem \ref{thm_bu}, the left-hand side is bounded from below by
$C \varepsilon$.
Therefore, we have
\begin{align*}%
	R \le C \varepsilon^{-\frac{2-\alpha}{4}(\frac{1}{p-1}-\frac{N-\alpha}{2})^{-1}}.
\end{align*}%
Since $R$ is arbitrary in $[1, T(\varepsilon))$, we obtain the desired estimate for $T(\varepsilon)$.

In the critical case $p=1+\frac{2}{N-\alpha}$, we define
\begin{align*}%
	Y(\rho)
	= \int_0^{\rho} \left( \iint_{\mathbb{R}^N\times (0,R)} |u|^p(1+t)\psi_R^{\ast}\,dxdt \right) R^{-1}\,dR.
\end{align*}%
We note that the changing variable
$s = \frac{\sqrt{|x|^{2-\alpha}+t^2}}{R}$
implies
\begin{align*}%
	Y(\rho) &\le \iint_{\mathbb{R}^N \times (0,\rho)} |u|^p (1+t)
				\left(
				\int_0^{\rho} \left[ \eta^{\ast}\left( \frac{|x|^{2-\alpha}+t^2}{R^2} \right) \right]^{2p'} R^{-1}\,dR
				\right)
			\,dxdt \\
		&= \iint_{\mathbb{R}^N \times (0,\rho)} |u|^p(1+t)
			\left( \int_{\frac{\sqrt{|x|^{2-\alpha}+t^2}}{\rho}}^{\infty}
				\left[ \eta^{\ast}(s^2) \right]^{2p'} s^{-1}\,ds \right)
			\,dxdt.
\end{align*}%
Here, noting
${\rm supp}\,\eta^{\ast} \subset [\frac12, 1]$,
$\eta^{\ast} \le \eta$
and $\eta$ is decreasing, we estimate
\begin{align*}%
	\int_{\frac{\sqrt{|x|^{2-\alpha}+t^2}}{\rho}}^{\infty}
				\left[ \eta^{\ast}(s^2) \right]^{2p'} s^{-1}\,ds
	\le \log 2 \left[ \eta\left( \frac{|x|^{2-\alpha}+t^2}{\rho^2} \right) \right]^{2p'}.
\end{align*}%
Hence, we obtain
\begin{align*}%
	Y(\rho) \le \log 2 \iint_{\mathbb{R}^N \times (0,\rho)} |u|^p(1+t) \psi_{\rho} \,dxdt.
\end{align*}%
Combining this with \eqref{3_ineq}, we have the differential inequality of $Y(R)$
\begin{align*}%
	\left( \varepsilon \int_{\mathbb{R}^N} (u_1(x) + (a(x)-1) u_0(x))\,dx + Y(R) \right)^p
	\le C R Y'(R)
\end{align*}%
for $R\ge 1$.
Noting $Y(1)=0$ and solving the above, we conclude
\begin{align*}%
	\log R \le C  \left( \varepsilon \int_{\mathbb{R}^N} (u_1(x) + (a(x)-1) u_0(x))\,dx \right)^{-(p-1)}.
\end{align*}%
Since $R$ is arbitrary in $[1, T(\varepsilon))$, we obtain the desired estimate for $T(\varepsilon)$.

\begin{remark}\label{rem_beta0}
In the case $\alpha < 0, \beta=0$,
we modify the definition of
$\psi_R (t,x)$ and $\psi_R^{\ast}(t,x)$
by
\begin{align*}%
	\psi_R(t,x) = \left[ \eta\left( \frac{|x|^{2-\alpha}+t}{R} \right) \right]^{2p'},\quad
	\psi_R^{\ast}(t,x) = \left[ \eta^{\ast}\left( \frac{|x|^{2-\alpha}+t}{R} \right) \right]^{2p'},
\end{align*}%
and we use
$\psi_R$ itself instead of $\Psi_R$.
Then, corresponding to \eqref{3_ineq}, we can prove
\begin{align*}%
	&\varepsilon \int_{\mathbb{R}^N} (u_1(x) + a(x) u_0(x))\,dx
		+ \iint_{\mathbb{R}^N\times (0,R)} |u|^p \psi_R\,dxdt \\
	&\le CR^{-\frac{2}{2-\alpha} (\frac{1}{p-1} - \frac{N-\alpha}{2})\frac{1}{p'}}
		\left( \iint_{\mathbb{R}^N\times (0,R)} |u|^p \psi_R^{\ast} \,dxdt\right)^{\frac{1}{p}}.
\end{align*}%
From this, one has the assertion of Theorem \ref{thm_bu} for $\alpha<0, \beta = 0$
in the same manner.
\end{remark}

%
%
%
%

\appendix
\section{Short proof of a special case of the Caffarelli--Kohn--Nirenberg inequality}
We give a short proof of Lemma \ref{CKN_ineq}.
For any $\gamma \ge 0$, by the integration by parts, we compute
\begin{align*}%
	\int_{\mathbb{R}^N} \nabla\cdot (|x|^{\gamma}x) u^2 \,dx
	&= -2 \int_{\mathbb{R}^N} |x|^{\gamma} ( x \cdot \nabla u) u \,dx\\
	&\le 2 \| \nabla u \|_{L^2} \| |x|^{\gamma+1} u \|_{L^2},
\end{align*}%
which leads to
\begin{align}%
\label{app_1}
	\| |x|^{{\gamma}/{2}} u \|_{L^2}^2 \le C(\gamma, N) \| \nabla u\|_{L^2} \| |x|^{\gamma+1} u \|_{L^2}.
\end{align}%
Taking $\gamma = 0$ and $\gamma =2$, we have
\begin{align}%
\label{app_2}
	\| u \|_{L^2}^2 \le C(N) \| \nabla u\|_{L^2} \| |x| u \|_{L^2}
\end{align}%
and
\begin{align}%
\label{app_3}
	\| |x| u \|_{L^2}^2 \le C(\gamma, N) \| \nabla u\|_{L^2} \| |x|^3 u \|_{L^2},
\end{align}%
respectively.
The inequality \eqref{app_2} gives the assertion for $k=1$.
Next, combining \eqref{app_2} and \eqref{app_3}, we deduce
\begin{align}%
\label{app_4}
	\| u \|_{L^2}^2 \le C(N) \| \nabla u\|_{L^2}^{3/4} \| |x|^3 u \|_{L^2}^{1/4},
\end{align}%
which gives the assertion for $k=2$.
Furthermore, taking $\gamma = 6$ in \eqref{app_1}, we have
\begin{align*}%
	\| |x|^3 u \|_{L^2}^2 \le C(N) \| \nabla u \|_{L^2} \| |x|^7 u \|_{L^2}.
\end{align*}%
This and \eqref{app_4} imply
\begin{align*}%
	\| u \|_{L^2} \le C(N) \| \nabla u \|_{L^2}^{7/8} \| |x|^7 u \|_{L^2}^{1/8},
\end{align*}%
which proves the assertion for $k=3$.
Repeating this argument, we obtain the desired estimate.


\section{Sketch of a priori estimate in case of time and space dependent coefficient damping}

Under Proposition \ref{prop_lwp},
we here sketch the proof of the following proposition, corresponding to Proposition \ref{2_prop_ap}, assuming
(i) $\alpha<0, \ -1<\beta<1$ or
(ii) $\alpha<0, \ \beta=1$ with
$a_0\gg 1$.
Then we can reach to Theorems \ref{thm_gwp}--\ref{thm_gwp_b1}, using usual procedure of the energy method.

\begin{proposition}\label{prop_app2}
Under the assumptions of Theorems \ref{thm_gwp}--\ref{thm_gwp_b1},
there exist
$\delta_0 \in (0,\frac{(N-\alpha)(1+\beta)}{2-\alpha})$,
and constants 
$t_0=t_0(a_0, R_0,\delta_0)\ge 1$,
$C=C(N,\alpha,\beta,a_0,R_0,\delta_0,t_0)$
such that the solution $u$ constructed in Proposition 1.1 satisfies the a priori estimate
\begin{align*}
   M^{(\beta)}(t) \le CM^{(\beta)}(0) + CM^{(\beta)}(t)^{\frac{p+1}{2}}
\end{align*}
for
$t \in [0,T(\varepsilon))$,
where
$\psi^{(\beta)}$
and
$M^{(\beta)}(t)$
are defined in \eqref{A5} and \eqref{A6} with
$\delta = \frac{2-\alpha}{2(N-\alpha)(1+\beta)}\delta_0$
below, respectively.
\end{proposition}

First, the function $\psi^{(\beta)}$ is defined by
\begin{align}
\label{A5}
	\psi^{(\beta)}(t,x) &=
		\frac{\mu}{B(t)}
		\left\{
			\langle x \rangle^{2-\alpha}+ A_0
			-N \ast \left( \alpha(2-\alpha)\langle x\rangle^{-2-\alpha} \eta_{R_{\delta}}(x) \right)
			\right\},
\end{align}
where
$B(t) = \frac{1}{1+\beta} + \int_0^t \frac{d\tau}{b(\tau)} = \frac{1}{1+\beta}(1+t)^{1+\beta}$.
Since $\psi^{(0)}$ satisfies \eqref{2_sowapsi_1}--\eqref{2_sowapsi_2} in Lemma \ref{2_lem_psi},
and since $\psi^{(\beta)}= \frac{1+t}{B(t)} \psi^{(0)}$
and $(\frac{1}{B(t)})'= - \frac{(1+t)^{\beta}}{B(t)^2}$,
\begin{align*}
	&-c(t,x)\psi_t^{(\beta)}
		= - c(t,x) \frac{(1+t)^{\beta}}{B(t)^2} (1+t)^2 \psi_t^{(0)}
		\ge - \left( \frac{1+t}{B(t)} \right)^2 a(x) \psi_t^{(0)} \\
	&\ge \left( \frac{1+t}{B(t)} \right)^2 (2+\delta_1^{(0)}) |\nabla \psi^{(0)}|^2
		= (2+\delta_1^{(0)})|\nabla \psi^{(\beta)}|^2
\end{align*}
and
\begin{align*}
	\Delta \psi^{(\beta)}
	&= \frac{1+t}{B(t)}\Delta \psi^{(0)}
	\ge \frac{1+\beta}{(1+t)^{\beta}}
		\left( \frac{N-\alpha}{2(2-\alpha)}-\delta_2^{(0)} \right) \frac{a(x)}{1+t}\\
	&= \left( \frac{(N-\alpha)(1+\beta)}{2(2-\alpha)}-\delta_2^{(\beta)} \right) \frac{c(t,x)}{1+t}.
\end{align*}
Here $\delta_1^{(\beta)}=\frac23 \delta$,
$\delta_2^{(\beta)}=\frac{(N-\alpha)(1+\beta)}{2(2-\alpha)}\delta$.
Hence, samely as in Lemma \ref{2_lem_psi}, it holds that
\begin{align*}%
	-\psi_t(t,x) c(t,x) &\ge (2+\delta_1^{(\beta)})|\nabla \psi(t,x)|^2, \\
	\Delta \psi(t,x) &\ge \left( \frac{(N-\alpha)(1+\beta)}{2(2-\alpha)}-\delta_2^{(\beta)} \right) \frac{c(t,x)}{1+t}.
\end{align*}%
Next, since the number $\frac{(N-\alpha)(1+\beta)}{2(2-\alpha)}$ suggests the decay rate, for the weighted energy estimate we define $M^{(\beta)}$ by
\begin{align}
\label{A6}
	M^{(\beta)}(t)
	&= \sup_{0<\tau<t}
		\left\{ (t_0+\tau)^{\frac{(N-\alpha)(1+\beta)}{2(2-\alpha)} + \beta +1 -\delta_0}
			\int_{{\mathbb R}^N} e^{2\psi}(u_t^2+ |\nabla u|^2)\,dx \right. \\
\nonumber
	&\qquad\qquad \left. + 
	(t_0+\tau)^{\frac{(N-\alpha)(1+\beta)}{2(2-\alpha)}-\delta_0} \int_{{\mathbb R}^N} e^{2\psi} a(x)u^2 \,dx \right\}, 
\end{align}
where $t_0$ is suitably large number and $\delta_0 \ge 4 \delta_2$. 

Here and after, we abbreviate the suffix
$(\beta)$ for $\psi^{(\beta)}$, $\delta^{(\beta)}$ and $M^{(\beta)}$.
Multiplying (1.1) by $e^{2\psi}u_t$ and $e^{2\psi}u$, we respectively have
\begin{align}
\label{A7}
	&\frac{\partial}{\partial t} \left[ \frac{e^{2\psi}}{2} (u_t^2+|\nabla u|^2) \right]
	- \nabla \cdot \left( e^{2\psi} u_t \nabla u \right) 
 	+ e^{2\psi} \left\{ \left( \frac14 c(t,x) -\psi_t \right) u_t^2 + \frac{-\psi_t}{5} |\nabla u|^2 \right\} \\
\nonumber
	&\le \frac{\partial}{\partial t} [e^{2\psi} F(u)] + 2e^{2\psi}(-\psi_t)F(u),
	\quad \mbox{where}\quad F(u)=\frac{1}{p+1} |u|^p u,
\end{align}
and
\begin{align}
\label{A8}
	&\frac{\partial}{\partial t} \left[ e^{2\psi} \left( u u_t + \frac{c(t,x)}{2} u^2 \right) \right]
		-\Delta \left( \frac{e^{2\psi}}{2} u^2 \right)
		+ e^{2\psi} \left( \delta_5 |\nabla u|^2 + \frac{\delta}{8}(-\psi_t)c(t,x)u^2 \right) \\
\nonumber
	&\quad + \left( \frac{(N-\alpha)(1+\beta)}{2(2-\alpha)} + \frac{\beta}{2} -\delta_2 \right)
			e^{2\psi}\frac{c(t,x)}{1+t}u^2
		- e^{2\psi} \left( 1+\frac{8(-\psi_t)}{\delta c(t,x)} \right) u_t^2 \\
\nonumber
 	&\le e^{2\psi} |u|^p u, 
\end{align}
since
$c(t,x)\frac{\partial}{\partial t}\frac{u^2}{2}
= \frac{\partial}{\partial t} \frac{c(t,x)}{2}u^2 + \frac{\beta c(t,x)}{2(1+t)} u^2$,
which is only different from the case
$\beta=0$,
though
$a(x)$
is changed to
$c(t, x)$.
So, note that
$\delta_i (i \ge 3)$
is the same as one in Section 2 (cf. Remark 2.2).
Integrating \eqref{A7} over ${\mathbb R}^N$ and multiplying it by
$(t_0+t)^{\frac{(N-\alpha)(1+\beta)}{2(2-\alpha)}+(1+\beta)-\delta_0}$, we have
\begin{align}
\label{A9}
	&\frac{d}{dt} \left[ (t_0+t)^{\frac{(N-\alpha)(1+\beta)}{2(2-\alpha)}+(1+\beta)-\delta_0}
		\int_{{\mathbb R}^N} \frac{e^{2\psi}}{2}(u_t^2 + |\nabla u|^2) \,dx \right] \\
\nonumber
	&\quad  - \left\{ \frac{(N-\alpha)(1+\beta)}{2(2-\alpha)}+(1+\beta)-\delta_0 \right\}
		(t_0+t)^{\frac{(N-\alpha)(1+\beta)}{2(2-\alpha)}+\beta-\delta_0}
			\int_{{\mathbb R}^N} \frac{e^{2\psi}}{2}(u_t^2+|\nabla u|^2)\,dx \\
\nonumber
	&\quad +(t_0+t)^{\frac{(N-\alpha)(1+\beta)}{2(2-\alpha)}+(1+\beta)-\delta_0}
		\int_{{\mathbb R}^N} e^{2\psi} \left\{ (\frac14 c(t,x) -\psi_t)u_t^2 +\frac{-\psi_t}{5}|\nabla u|^2 \right\} \,dx \\
\nonumber
	&\le \frac{d}{dt} \left[ (t_0+t)^{\frac{(N-\alpha)(1+\beta)}{2(2-\alpha)}+(1+\beta)-\delta_0}
		\int_{{\mathbb R}^N} e^{2\psi} F(u)\,dx \right] \\
\nonumber
	&\quad - \left\{
		\frac{(N-\alpha)(1+\beta)}{2(2-\alpha)}+(1+\beta)-\delta_0 \right\}
		(t_0+t)^{\frac{(N-\alpha)(1+\beta)}{2(2-\alpha)}+\beta-\delta_0}
		\int_{{\mathbb R}^N} e^{2\psi} F(u)\,dx \\
\nonumber
	&\quad  + 2(t_0+t)^{\frac{(N-\alpha)(1+\beta)}{2(2-\alpha)}+(1+\beta)-\delta_0}
		\int_{{\mathbb R}^N} e^{2\psi} (-\psi_t)F(u)\,dx.
\end{align}
Since 
\begin{align}
\label{A10}
	(t_0+t)c(t,x) &= (t_0+t) a_0\langle x \rangle^{-\alpha}(1+t)^{-\beta} \ge a_0t_0^{1-\beta_+} \gg 1,
\end{align}
if (i) or (ii) holds, the term
$\int_{{\mathbb R}^N}\frac{e^{2\psi}}{2} u_t^2$
in the second term in \eqref{A9} is absorbed into
$\int_{{\mathbb R}^N} e^{2\psi}\frac14 c(t,x)u_t^2\,dx$,
where
$\beta_+ = 0\ (\beta\le 0), \, \beta_+=\beta\ (0<\beta \le 1)$.
Hence, integration of \eqref{A9} over $[0,t]$ yields 
\begin{align}
\label{A11}
	&(t_0+t)^{\frac{(N-\alpha)(1+\beta)}{2(2-\alpha)}+(1+\beta)-\delta_0}
		\int_{{\mathbb R}^N} \frac{e^{2\psi}}{2} (u_t^2+|\nabla u|^2)\,dx \\
\nonumber
	&\quad +\int_0^t (t_0+\tau)^{\frac{(N-\alpha)(1+\beta)}{2(2-\alpha)}+(1+\beta)-\delta_0}
		\int_{{\mathbb R}^N} e^{2\psi}
			\left\{ \left( \frac18 c(\tau,x) -\psi_t \right) u_t^2 + \frac{-\psi_t}{5} |\nabla u|^2 \right\} \, dx \,d\tau \\
\nonumber
	&\quad - \left\{\frac{(N-\alpha)(1+\beta)}{2(2-\alpha)}+(1+\beta)-\delta_0 \right\}
		\int_0^t (t_0+\tau)^{\frac{(N-\alpha)(1+\beta)}{2(2-\alpha)}+(1+\beta)-\delta_0}
			\int_{{\mathbb R}^N} \frac{e^{2\psi}}{2} |\nabla u|^2\,dx \,d\tau \\
\nonumber
	&\le C(t_0)\int_{{\mathbb R}^N} e^{2\psi(0,x)} (u_1^2+|\nabla u_0|^2 + |u_0|^{p+1})\,dx \\
\nonumber
	&\quad + (t_0+t)^{\frac{(N-\alpha)(1+\beta)}{2(2-\alpha)}+(1+\beta)-\delta_0}
		\int_{{\mathbb R}^N} e^{2\psi} F(u)\,dx \\
\nonumber
	&\quad - \left\{ \frac{(N-\alpha)(1+\beta)}{2(2-\alpha)}+(1+\beta)-\delta_0 \right\}
		\int_0^t(t_0+\tau)^{\frac{(N-\alpha)(1+\beta)}{2(2-\alpha)}+\beta-\delta_0}
			\int_{{\mathbb R}^N} e^{2\psi} F(u)\,dx \,d\tau \\
\nonumber
	&\quad  + 2 \int_0^t (t_0+\tau)^{\frac{(N-\alpha)(1+\beta)}{2(2-\alpha)}+(1+\beta)-\delta_0}
		\int_{{\mathbb R}^N} e^{2\psi} (-\psi_t) F(u)\,dx \,d\tau.
\end{align}
Similarly, integrating \eqref{A8} over
${\mathbb R}^N \times [0,t]$, we have
\begin{align}
\label{A12}
	&(t_0+t)^{\frac{(N-\alpha)(1+\beta)}{2(2-\alpha)}+\beta-\delta_0}
		\int_{{\mathbb R}^N} e^{2\psi} \left( u u_t + \frac{c(t,x)}{2} u^2 \right)\,dx \\
\nonumber
	&\quad + \int_0^t (t_0+\tau)^{\frac{(N-\alpha)(1+\beta)}{2(2-\alpha)}+\beta-\delta_0}
		\int_{{\mathbb R}^N} e^{2\psi}
				\left( \delta_5 |\nabla u|^2 + \frac{\delta}{8} (-\psi_t)c(\tau,x)u^2 \right) \,dx \,d\tau \\
\nonumber
	&\quad +\delta_6 \int_0^t (t_0+\tau)^{\frac{(N-\alpha)(1+\beta)}{2(2-\alpha)}+\beta-1-\delta_0}
		\int_{{\mathbb R}^N} e^{2\psi} c(\tau,x)u^2 \,dx \,d\tau \\
\nonumber
	&\quad -(1+\frac{C(\delta_6)}{a_0 t_0^{1-\beta_+}})
		\int_0^t (t_0+\tau)^{\frac{(N-\alpha)(1+\beta)}{2(2-\alpha)}+\beta-\delta_0}
		\int_{{\mathbb R}^N} e^{2\psi} u_t^2\,dx \,d\tau \\
\nonumber
	&\quad -\frac{8}{\delta} \int_0^t (t_0+\tau)^{\frac{(N-\alpha)(1+\beta)}{2(2-\alpha)}+\beta-\delta_0}
		\int_{{\mathbb R}^N} e^{2\psi} \frac{-\psi_t}{c(\tau,x)} u_t^2 \,dx \,d\tau \\
\nonumber
	&\le C(\alpha,\beta,t_0) \int_{{\mathbf R}^N} e^{2\psi(0,x)} \left( u_0 u_1 + \frac{c(0,x)}{2}u_0^2 \right)\,dx \\
\nonumber
	&\quad + \int_0^t (t_0+\tau)^{\frac{(N-\alpha)(1+\beta)}{2(2-\alpha)}+\beta-\delta_0}
		\int_{{\mathbb R}^N} e^{2\psi} |u|^pu \,dx \,d\tau, 
\end{align}
where we used, by \eqref{A10},
\begin{align*}
	\int_{{\mathbb R}^N} e^{2\psi} u u_t \,dx
	&\le \delta_6 \int_{{\mathbb R}^N} e^{2\psi} c(\tau,x)u^2 \,dx
		+ (t_0+\tau)\cdot \frac{C(\delta_0)}{(t_0+\tau)c(\tau,x)}\int_{{\mathbb R}^N} e^{2\psi} u_t^2\,dx \\
	&\le \delta_6 \int_{{\mathbb R}^N} e^{2\psi} c(\tau,x)u^2 \,dx
		+ (t_0+\tau)\cdot \frac{C(\delta_6)}{a_0t_0^{1-\beta_+}} \int_{{\mathbb R}^N} e^{2\psi} u_t^2 \,dx. 
\end{align*}
Now, we add \eqref{A11} to
$\nu \cdot$\eqref{A12} ($0<\nu \ll 1$)
and cover the bad terms. In fact, fix $\nu$ small as
$\nu \cdot \{ \frac{(N-\alpha)(1+\beta)}{2(2-\alpha)}+(1+\beta)-\delta_0\} \le \frac{\delta_5}{2}$,
and then take $t_0$ or $a_0$ large in case of (i) or (ii) as
\begin{align*}
	1+\frac{C(\delta_6)}{a_0t_0^{1-\beta_+}} \le \frac{\nu}{16}(t_0+\tau)c(\tau,x)
	\quad \mbox{and} \quad \frac{8}{\delta}\cdot \frac{1}{c(\tau,x)} \le \frac{\nu}{2}(t_0+\tau),
\end{align*}
that is, by \eqref{A10},
\begin{align*}
	1+\frac{C(\delta_6)}{a_0t_0^{1-\beta_+}}
		&\le \frac{\nu}{16}a_0t_0^{1-\beta_+} \le \frac{\nu}{16} (t_0+\tau) c(\tau,x)
	\quad \mbox{and} \quad \frac{8}{\delta} \le \frac{\nu}{2}a_0t_0^{1-\beta_+}
		\le \frac{\nu}{2}(t_0+\tau)c(\tau,x).
\end{align*}
Thus, the bad terms in the left hand sides in
\eqref{A11}--\eqref{A12} are absorbed to good ones and the following desired inequality holds:
\begin{align*}
	&\nu (t_0+t)^{\frac{(N-\alpha)(1+\beta)}{2(2-\alpha)}+(1+\beta)-\delta_0}
		\int_{{\mathbb R}^N} \frac{e^{2\psi}}{2} (u_t^2+|\nabla u|^2)\,dx \\
	&\quad + (t_0+t)^{\frac{(N-\alpha)(1+\beta)}{2(2-\alpha)}+\beta-\delta_0}
		\int_{{\mathbb R}^N} e^{2\psi} (u u_t + \frac{c(t,x)}{2}u^2) \,dx \\
	&\quad +\nu \int_0^t (t_0+\tau)^{\frac{(N-\alpha)(1+\beta)}{2(2-\alpha)}+(1+\beta)-\delta_0}
		\int_{{\mathbb R}^N} e^{2\psi}
			\left\{ (\frac{1}{16}c(\tau,x)-\frac12 \psi_t)u_t^2 + \frac{-\psi_t}{5}|\nabla u|^2 \right\} \,dx \,d\tau  \\
	&\quad + \int_0^t (t_0+\tau)^{\frac{(N-\alpha)(1+\beta)}{2(2-\alpha)}+\beta-\delta_0} 
		\int_{{\mathbb R}^N} e^{2\psi}
			\left( \frac{\delta_6}{2}|\nabla u|^2 + \frac{\delta}{8}(-\psi_t)c(\tau,x) u^2 \right) \,dx \,d\tau  \\
	&\le C(t_0)\int_{{\mathbb R}^N} e^{2\psi(0,x)}
			\left( u_1^2+|\nabla u_0|^2 +|u_0|^{p+1} + \frac{c(0,x)}{2}u_0^2 \right) \,dx  \\
	&\quad +\nu (t_0+t)^{\frac{(N-\alpha)(1+\beta)}{2(2-\alpha)}+(1+\beta)-\delta_0} 
		\int_{{\mathbf R}^N} e^{2\psi} |F(u)|\,dx \\
	&\quad +C(t_0) \int_0^t (t_0+\tau)^{\frac{(N-\alpha)(1+\beta)}{2(2-\alpha)}+\beta-\delta_0} 
	\int_{{\mathbb R}^N} e^{2\psi} |u|^{p+1}\,dx \,d\tau \\
	&\quad +2\nu \int_0^t (t_0+\tau)^{\frac{(N-\alpha)(1+\beta)}{2(2-\alpha)}+(1+\beta)-\delta_0}
	\int_{{\mathbb R}^N} e^{2\psi} (-\psi_t)|F(u)|\,dx \,d\tau. 
\end{align*}
For the semilinear terms, we estimate, for an example, 
\begin{align*}
	N_1 &:= \nu (t_0+t)^{\frac{(N-\alpha)(1+\beta)}{2(2-\alpha)}+(1+\beta)-\delta_0}
			\int_{{\mathbb R}^N} e^{2\psi} |F(u)|\,dx  \\
		&= \nu (t_0+t)^{\frac{(N-\alpha)(1+\beta)}{2(2-\alpha)}+(1+\beta)-\delta_0}
			\Vert  e^{\frac{2\psi}{p+1}} u \Vert_{L^{p+1}}^{p+1}. 
\end{align*}
By the Gagliardo-Nirenberg inequality
\begin{align*}
	\Vert e^{\frac{2\psi}{p+1}} u \Vert_{L^{p+1}}
		&\le C \Vert \nabla (e^{\frac{2\psi}{p+1}} u) \Vert_{L^2}^{\theta} 
			\Vert e^{\frac{2\psi}{p+1}} u \Vert_{L^2}^{1-\theta},
				\quad \mbox{where} \quad \theta= \frac{N(p-1)}{2(p+1)}, \\
		&\le C(\Vert e^{\frac{2\psi}{p+1}}
			|\nabla \psi| u \Vert_{L^2} + \Vert e^{\frac{2\psi}{p+1}} \nabla u \Vert_{L^2})^{\theta}
				\Vert e^{\frac{2\psi}{p+1}} u \Vert_{L^2}^{1-\theta}.
\end{align*}
Since
$\psi=\psi^{(\beta)}$ has the $t$-dependent coefficient $\frac{\mu}{B(t)}$,
instead of $\frac{\mu}{1+t}$ for $\psi=\psi^{(0)}$ in Section 2,
by the Caffarelli--Korn--Nirenberg inequality,
samely as in Lemmas \ref{2_lem_int}--\ref{2_lem_int2}, we have
\begin{align*}
	&\Vert e^{\frac{2\psi}{p+1}} u \Vert_{L^{p+1}} \\
	&\le C(B(t)^{-\frac12} \Vert e^{\psi}\sqrt{a}u \Vert_{L^2}
		+ \Vert e^{\psi}\nabla u \Vert_{L^2})^{\theta + (1-\frac{1}{2^k})(1-\theta)}
			\times B(t)^{\frac{2^k}{2-\alpha}-\frac12} \Vert e^{\psi} \sqrt{a}u \Vert_{L^2}^{\frac{1}{2^k}(1-\theta)}.
\end{align*}
Hence
\begin{align*}
	N_1
	&\le C(t_0+t)^{\frac{(N-\alpha)(1+\beta)}{2(2-\alpha)}+(1+\beta)-\delta_0}
		\times
		\left[(t_0+t)^{%
			-\frac12 (\frac{(N-\alpha)(1+\beta)}{2(2-\alpha)}+(1+\beta)-\delta_0)(\theta +(1-\frac{1}{2^k})(1-\theta))}
		\right. \\
	&\qquad \left. \cdot (t_0+t)^{%
			\{ (1+\beta)(\frac{2^k}{2-\alpha}-\frac12 -\frac12 (\frac{(N-\alpha)(1+\beta)}{2(2-\alpha)}-\delta_0)\}%
			\frac{1}{2^k}(1-\theta)}\right]^{p+1}
			\times M(t)^{\frac{p+1}{2}}.
\end{align*}
The exponent of $(t_0+t)$ is
\begin{align*}
	&(1+\beta) \left[ \frac{N-\alpha}{2-\alpha}+1-\frac{\delta_0}{1+\beta}
		+ (p+1) \left\{ -\frac12 \left( \frac{N-\alpha}{2-\alpha}+1-\frac{\delta_0}{1+\beta} \right)
		\left( \theta+ \left( 1-\frac{1}{2^k}\right) (1-\theta) \right) \right.\right. \\
	&\quad \left.\left. + \left( \left(
		\frac{2^k}{2-\alpha}-\frac12 \right) -\frac12 \left(\frac{N-\alpha}{2-\alpha}-\frac{\delta_0}{1+\beta}
		\right) \right)
		\frac{1-\theta}{2^k} \right\} \right].  
\end{align*}
Samely as \eqref{2_p}, this is negative if and only if
\begin{align*}
	p>1+ \frac{2}{N-\alpha - (2-\alpha)\delta_0/2(1+\beta)},
\end{align*}
because of 
$1+\beta>0$.
The other semilinear terms are estimated in a similar fashion to the above.
Thus, taking $\delta_0>0$ small, we obtain the desired estimate on $M(t) = M^{(\beta)}(t)$ if
$p>1+\frac{2}{N-\alpha}$, which completes
Proposition \ref{prop_app2} for $\beta \ne 0$.

\section*{Acknowledgement}
This work was supported by JSPS KAKENHI Grant Number JP18K134450 and JP16K17625.

\end{document}